\documentclass[12pt]{amsart}
\usepackage[a4paper,margin=2.5cm]{geometry}
\usepackage{amsthm,amsmath,amssymb,amsfonts}

\usepackage{mathtools}
\usepackage{latexsym}
\usepackage{bm}
\usepackage[all]{xy}

\usepackage{hyperref}

\usepackage{graphicx}
\usepackage{xcolor}
\usepackage[percent]{overpic}

\usepackage{libertinus}
\begin{document}
\theoremstyle{plain}
\newtheorem{thm}{Theorem}[section]
\newtheorem{prop}[thm]{Proposition}
\newtheorem{lem}[thm]{Lemma}
\newtheorem{cor}[thm]{Corollary}

\theoremstyle{definition}
\newtheorem{defi}[thm]{Definition}
\newtheorem{rmk}[thm]{Remark}
\newtheorem{exm}[thm]{Example}
\newtheorem{que}[thm]{Question}
\newtheorem{prob}[thm]{Ploblem}

\title{On the Ozsv\'ath-Szab\'o $d$-invariants for almost simple linear graphs}    
\author{Tatsumasa Suzuki}
\thanks{The author is partially supported by JST SPRING, Grant Number JPMJSP2106}

\address{Department of Frontier Media Science, Meiji University, 4-21-1 Nakano, Nakano-ku, Tokyo 164-8525, Japan}
\email{suzukit519@meiji.ac.jp}

\date{\today}

\begin{abstract}  
Karakurt and \c{S}avk computed the Ozsv\'ath-Szab\'o $d$-invariants of Brieskorn homology $3$-spheres arising as surgeries on almost simple linear graphs. 
In this paper, we refine their formula for these $d$-invariants. 
Furthermore, we present infinite families of examples for which the inequalities appearing in this refinement are equalities, as well as infinite families for which they are strict.
As an application, we derive consequences for the knot concordance group. 
\end{abstract}

\maketitle
\tableofcontents

\section{Introduction}
The $n$-dimensional homology cobordism group $\Theta_{\mathbb{Z}}^n$ was introduced by Gonz\'alez-Acu\~na~\cite{GAn70}, building on earlier work of Kervaire and Milnor~\cite{KM63} on the $n$-dimensional homotopy cobordism group $\Theta^n$. 

Matumoto~\cite{Mat78}, and Galewski and Stern~\cite{GS80} related the triangulation conjecture to the Rokhlin invariant and the $3$-dimensional homology cobordism group $\Theta_{\mathbb{Z}}^3$, reducing the conjecture to a problem concerning the interaction between $3$- and $4$-manifolds. 
Later, Manolescu~\cite{Man16} disproved the triangulation conjecture using $\mathrm{Pin}(2)$-equivariant Seiberg-Witten Floer homology, further highlighting the importance of $\Theta_{\mathbb{Z}}^3$ in low-dimensional topology; see also~\cite{Man18}. 

Rokhlin~\cite{Rok52} proved that the $3$-dimensional homology cobordism group $\Theta_{\mathbb{Z}}^3$ is nontrivial.
Fintushel and Stern~\cite{FS85} later showed that $\Theta_{\mathbb{Z}}^3$ contains a $\mathbb{Z}$-subgroup, and Furuta~\cite{Fur90} proved that it contains a $\mathbb{Z}^{\infty}$-subgroup. 

On the other hand, Fr{\o}yshov~\cite{Fro02} showed that $\Theta_{\mathbb{Z}}^3$ admits a $\mathbb{Z}$-summand. 
Later, Dai, Hom, Stoffregen, and Truong~\cite{DHST23} constructed a $\mathbb{Z}^{\infty}$-summand generated by a family of Brieskorn homology $3$-spheres with almost simple linear graphs using the Ozsv\'ath-Szab\'o $d$-invariant~\cite{OS03}. 
Moreover, Karakurt and \c{S}avk~\cite{KS22} exhibited another $\mathbb{Z}^{\infty}$-summand generated by two distinct families of Brieskorn homology $3$-spheres with almost simple linear graphs, based on computations of the $d$-invariant carried out in~\cite{KS20}. 

The group $\Theta_{\mathbb{Z}}^3$ has also been studied from the viewpoints of both low-dimensional topology and knot theory through its interaction with the knot concordance group $\mathcal{C}$.
For example, by adapting gauge-theoretic methods introduced in~\cite{Fur90}, Endo~\cite{End95} proved that the subgroup $\mathcal{C}_{\mathrm{TS}} $ of $\mathcal{C}$ generated by topologically slice knots contains a $\mathbb{Z}^{\infty}$-subgroup. 
On the other hand, Dai, Hom, Stoffregen, and Truong~\cite{DHST23} established the existence of a $\mathbb{Z}^{\infty}$-summand in $\Theta_{\mathbb{Z}}^3$, following ideas related to Hom's method that $\mathcal{C}_{\mathrm{TS}}$ admits a $\mathbb{Z}^{\infty}$-summand~\cite{Hom15}.

Despite the development of Floer-theoretic invariants, the structure of the homology cobordism group $\Theta_{\mathbb{Z}}^3$ remains largely unknown. 
A central difficulty is the lack of explicit computations for large natural families of homology $3$-spheres. 

The Ozsv\'ath-Szab\'o $d$-invariant $d(Y,\mathfrak{s})$ is a rational homology $\text{spin}^c$ cobordism invariant assigning a rational number to a rational homology $3$-sphere $Y$ equipped with a $\text{spin}^c$ structure $\mathfrak{s}$; see~\cite{OS03} for details.
Since a homology $3$-sphere admits a unique $\text{spin}^c$ structure, we suppress it from the notation.
In particular, if $Y$ is a homology $3$-sphere, then $d(Y)$ is an even integer and is invariant under homology cobordism.

Let $p$, $q$, and $r$ are pairwise relatively prime positive integers.
The Brieskorn homology $3$-sphere $\Sigma(p,q,r)$ is defined by
\[
\Sigma(p,q,r)
=
\{(x,y,z)\in\mathbb{C}^3 \mid x^p+y^q+z^r=0\}
\cap S_{\varepsilon}^5,
\]
where $S_{\varepsilon}^5$ denotes the $5$-sphere of sufficiently small radius $\varepsilon>0$.
The homology sphere $\Sigma(p,q,r)$ is unchanged up to diffeomorphism under permutations of $p$, $q$, and $r$, and is diffeomorphic to $S^3$ whenever one of $p$, $q$, or $r$ is equal to $1$. 

We further assume that $1<p<q<r$ are integers satisfying
\[
pq+pr-qr=1.
\]

A surgery diagram of the closed $3$-manifold $\Sigma(p,q,r)$ is given by the plumbed $3$-manifold shown at the left of Figure~\ref{fig(ASL-graph)}, which corresponds to the weighted graph displayed at the right of Figure~\ref{fig(ASL-graph)}.

\begin{figure}[htbp]

\vspace{5truemm}

    \centering
    \begin{overpic}[scale=0.5
    ]{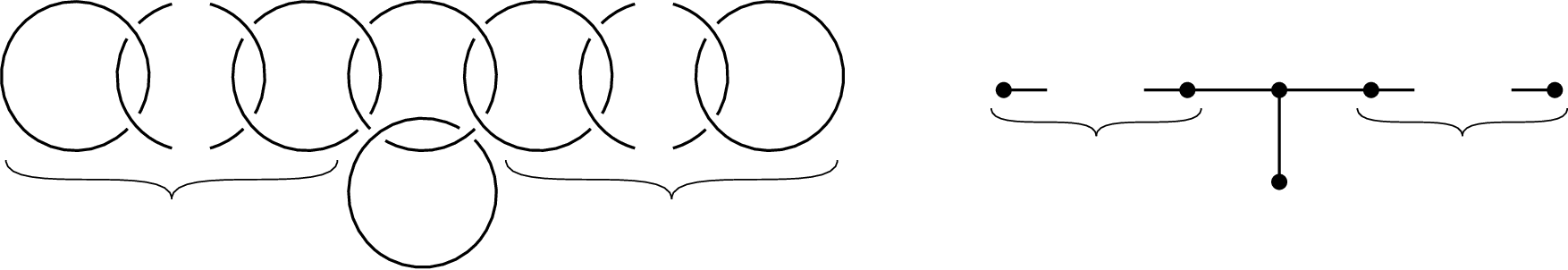}
    \put(2,18){$-2$}
    \put(10.5,11.5){$\cdots$}
    \put(18,18){$-2$}
    \put(25,18){$-2$}
    \put(33,18){$-2$}
    \put(40,11.5){$\cdots$}
    \put(48,18){$-2$}

    \put(8,1.5){$q-1$}
    \put(40,1.5){$r-1$}
    \put(21,-2){$-p$}

    \put(57,11){$=$}

    \put(62,13){$-2$}
    \put(68,10.5){$\cdots$}
    \put(73,13){$-2$}
    \put(80,13){$-2$}
    \put(86,13){$-2$}
    \put(91.5,10.5){$\cdots$}
    \put(97,13){$-2$}

    \put(67,6){$q-1$}
    \put(90.5,6){$r-1$}
    \put(80,2.5){$-p$}
    \end{overpic}

\vspace{5truemm}

\caption{A surgery diagram of $\Sigma(p,q,r)$ with $pq+pr-qr=1$}
\label{fig(ASL-graph)}
\end{figure}

We call this weighted graph an \textit{almost simple linear graph}. 
The most classical example among such Brieskorn homology $3$-spheres is the Poincar\'e homology $3$-sphere $\Sigma(2,3,5)$, whose almost simple linear graph coincides with the $E_8$ plumbing graph. 

Karakurt and \c{S}avk~\cite{KS20} studied the Ozsv\'ath-Szab\'o $d$-invariant $d(\Sigma(p,q,r))$ in order to investigate the behavior of these manifolds in the homology cobordism group $\Theta_{\mathbb{Z}}^3$. 
They computed the invariant explicitly in the case where $p$ is even. 

We fix the orientation of $\Sigma(p,q,r)$ so that $d(\Sigma(p,q,r))\ge0$. 

\begin{prop}[Karakurt-\c{S}avk~{\cite[Proposition~4.5]{KS20}}]
\label{prop(even-case)}
Suppose that $p$ is even and $pq+pr-qr=1$. 
Then
\[
d(\Sigma(p,q,r))=\frac{q+r}{4}=\frac{r^2-1}{4(r-p)}=\frac{q^2-1}{4(q-p)}. 
\]
\end{prop}

They also derived a formula for the case where $p$ is odd; see~\cite{KS22} or Section~\ref{sec(Preliminaries)}. 
However, explicit closed formulas for $d(\Sigma(p,q,r))$ remain unknown when $p$ is odd, even for families defined by almost simple linear graphs. 

We now give a more explicit description in the case where $p$ is odd. 
Set $n_p=(p-1)/2$. 
Let $t_{p,q}$ and $\alpha_{p,q}$ denote the quotient and remainder obtained by dividing $n_p$ by $q-p$, that is,
\[
n_p = (q-p)t_{p,q}+\alpha_{p,q},
\qquad
0\le \alpha_{p,q} < q-p .
\]

Define the quadratic function
\[
F_{p,q}(x,y)
:=
\frac{-(q+r)x^2+4qxy-4(q-p)y^2-4y+q+r}{4}.
\]

Define the set
\[
\mathfrak{M}_{p,q}
:=
\{(1,t_{p,q}+1)\}
\sqcup
\{(a,m)\in\mathbb{Z}^2
\mid
a\in2\mathbb{N}+1,\;
a<m\le n_p,\;
F_{p,q}(a,m)\ge F_{p,q}(1,1)
\},
\]
where $\mathbb{N}$ denotes the set of positive integers.

\begin{thm}
\label{thm(odd-case)}
Suppose that $p$ is odd and $pq+pr-qr=1$.
Then
\begin{align*}
d(\Sigma(p,q,r))
&=
\max_{(a,m)\in\mathfrak{M}_{p,q}} F_{p,q}(a,m)\\
&\ge
\max_{1\le m\le n_p} F_{p,q}(1,m)
=
F_{p,q}(1,t_{p,q}+1)
=
(t_{p,q}+1)(n_p+\alpha_{p,q}).
\end{align*}
\end{thm}

\begin{rmk}
If $t_{p,q}$ is odd, then $t_{p,q}+1$ is even. 
If $t_{p,q}$ is even, then $n_p+\alpha_{p,q}=t_{p,q}(q-p)+2\alpha_{p,q}$ is even. 
Hence $(t_{p,q}+1)(n_p+\alpha_{p,q})$ is always an even integer. 
\end{rmk}

Using Theorem~\ref{thm(odd-case)}, we obtain several new cases.

\begin{thm}
\label{thm(alpha0case)}
Suppose that $p$ is odd, $pq+pr-qr=1$, and $\alpha_{p,q} = 0$, then
\[
d(\Sigma(p,q,r)) = (t_{p,q}+1)n_p.
\]
\end{thm}

\begin{thm}
\label{thm(1-19)}
Suppose that $p$ is odd, $pq+pr-qr=1$, and $1 \le q-p \le 19$. 
Then
\[
d(\Sigma(p,q,r))=D(p,q,r).
\]
\end{thm}

Define
\[
D(p,q,r) :=
\begin{cases}
(t_{p,q}+1)(n_p+\alpha_{p,q}), & \text{if $p$ is odd}, \\[2mm]
d(\Sigma(p,q,r)), & \text{if $p$ is even}.
\end{cases}
\]

Let $(p_i,q_i,r_i)\in\mathbb{Z}^3$ satisfy $1<p_i<q_i<r_i$, and $p_iq_i+p_ir_i-q_ir_i=1  \quad (i=1,2)$.

The following inequality holds independently of the parity of $p$.

\begin{prop}
\label{prop(D-ineq)}
Let $p_i,q_i$, and $r_i$ be integers with $1<p_i<q_i<r_i$ satisfying 
\[
p_iq_i+p_ir_i-q_ir_i=1 \quad (i=1,2).
\]
Assume that $p_1=p_2=p$ and $q_1 \ge q_2$. 
Then
\[
2\left\lfloor\frac{p}{2}\right\rfloor
\le 
D(p_1,q_1,r_1)
\le 
D(p_2,q_2,r_2)
\le 
\left\lfloor\frac{p}{2}\right\rfloor^2
+
\left\lfloor\frac{p}{2}\right\rfloor.
\]
\end{prop}

Now assume that $p$ is odd.
Define $l_{p,q}:=q-p$ and $s_{p,q}:=l_{p,q}/n_p$. 

\begin{thm}
\label{thm(d=D=p-1-condition)}
Let $p$ be an odd integer. 
Suppose that $pq + pr - qr = 1$ and 
\[
s_{p,q}=\frac{2u}{u+1}
\]
for some positive integer $u$. 
Then $d(\Sigma(p,q,r))=p-1$. 
\end{thm}
We also exhibit an infinite family for which $d(\Sigma(p,q,r)) = D(p,q,r)$ (See Section~\ref{sec(infinite-kinds-of-infinite-classes)}). 

In this paper, we present a sufficient condition for the inequality
\[
d(\Sigma(p,q,r)) \neq D(p,q,r)
\]
to hold.

\begin{thm}
\label{thm(D-neq-d-condition)}
Let $p$ be a sufficiently large odd integer satisfying $pq + pr - qr = 1$. 
Assume that $s_{p,q} \in (1,2)$ and 
\[
s_{p,q} \neq \frac{2u}{u+1}
\]
for any positive integer $u$. 
Then $d(\Sigma(p,q,r)) > D(p,q,r) = p - 1$. 
\end{thm}

We also exhibit an infinite family for which $d(\Sigma(p,q,r)) \neq D(p,q,r)$ (see Section~\ref{sec(d-neq-D)}). 

Let $F_m$ denote the $m$-th Fibonacci number. 
The $d$-invariant of the Brieskorn homology $3$-sphere $\Sigma(F_{2k+1},F_{2k+2},F_{2k+3})$ with $F_{2k+1}$ odd is the most complicated example considered in~\cite{KS20}, and it is the only example for which no explicit formula is known. 
We obtain the following nontrivial estimate. 

\begin{prop}
\label{prop(Fibonacci-case)}
If $k>4$ and $2k+1\notin 3\mathbb Z$, then
\[
d(\Sigma(F_{2k+1},F_{2k+2},F_{2k+3}))
\ge
F_{F_{2k+1},F_{2k+2}}(3,4)
>
D(F_{2k+1},F_{2k+2},F_{2k+3}).
\]
\end{prop}

Here we summarize the results of \cite{KS20} together with those obtained in this paper for the $d$-invariant $d(\Sigma(p,q,r))$ in the case where $p$ is odd and $pq+pr-qr=1$ in Figure~\ref{fig(summary)}. 
Note that the integer $l_{p,q}=q-p$ satisfies $1\le l_{p,q}\le p-1$. 

\begin{figure}[htbp]

\vspace{5truemm}

\centering
\begin{overpic}[scale=0.5]{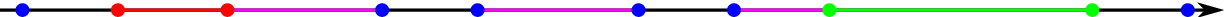}
\put(0.5,-4){\color{blue}{$1$}}
\put(8.5,-4){\color{red}{$2$}}
\put(17,-4){\color{red}{$19$}}

\put(23,4){\color{blue}{$(p+1)/4$}}
\put(33,-4){\color{blue}{$(p+3)/4$}}

\put(44,4){\color{blue}{$(p-1)/2$}}
\put(54,-4){\color{blue}{$(p+1)/2$}}

\put(64,4){\color{green}{$(p+3)/2$}}

\put(84,-4){\color{green}{$p-2$}}

\put(92,4){\color{blue}{$p-1$}}

\put(102,0){$l_{p,q}$}
\end{overpic}

\vspace{5truemm}

\caption{
Summary of results on the $d$-invariant $d(\Sigma(p,q,r))$ for odd $p$ satisfying $pq+pr-qr=1$
}
\label{fig(summary)}
\end{figure}
In Figure~\ref{fig(summary)}, the blue region indicates the range where $d(\Sigma(p,q,r))$ can be computed using the results of \cite{KS20}, while the red region corresponds to the range where $d(\Sigma(p,q,r))$ can be computed using Proposition~\ref{thm(1-19)}. 
The green region represents the range where there exist examples satisfying $d(\Sigma(p,q,r))\neq D(p,q,r)$. 
The green and pink regions remain largely unexplored, except for the case $\alpha_{p,q}=0$ (Theorem~\ref{thm(alpha0case)}). 
The figure illustrates that, in addition to Corollary~\ref{cor(Fibonacci-cases)}, the present study produces new pairs of Brieskorn homology $3$-spheres that are not homology cobordant. 

These results reveal a previously unseen transition between computable and irregular regimes of the $d$-invariant.

Levine and Lidman \cite{LL19} have used the calculation of the $d$-invariant $d(\Sigma(2,q,r))$ to detect a simply connected spineless $4$-manifold $X$ (i.e., the sphere $S^2$ embedded in $X$ in all piecewise linear $X$ is not homotopy equivalent to $X$). 
It is unknown whether $\Sigma(p,q,r)$ is diffeomorphic to a Dehn surgery $S^3_r(K) \quad (r\in\mathbb{Q}\cup\{\infty\})$ along a knot $K$ in $S^3$. 
To the best of the author's knowledge, there are almost no explicit computations of $d(\Sigma(p,q,r))$ for $p>5$, and no examples are known in which $\Sigma(p,q,r)$ is realized as a Dehn surgery along a knot in $S^3$, except for the results in~\cite{KS20}. 
Further development of the computation of $d(\Sigma(p,q,r))$ for almost simple linear graphs is expected to contribute to applications in the study of $4$-manifolds. 

Iida and the first author~\cite{IS25} showed that several infinite families of $P(p,q,r)$ are not squeezed by applying $\mathbb{Z}_2$-equivariant monopole Floer theory and Heegaard Floer theory, using the fact that the double branched covers $\Sigma_2(P(p,q,r))$ admit almost simple linear graphs. 
It is expected that the computational method for the $d$-invariant developed in this paper can also be applied to such rational homology $3$-spheres $\Sigma_2(P(p,q,r))$, potentially leading to further contributions to knot theory. 

\section*{Organization}
In Section~\ref{sec(Preliminaries)}, we review Brieskorn homology $3$-spheres (Subsection~\ref{subsec(BrieskornS3)}) and the Ozsv\'ath-Szab\'o $d$-invariant (Subsection~\ref{subsec(d-inv)}). 
We also introduce several related results concerning the homology cobordism group (Subsection~\ref{subsec(homology-cobordism-group)}) and the knot concordance group (Subsection~\ref{subsec(knot-concordance-group)}). 
In Section~\ref{sec(refinement)}, we prove the main result, which can be regarded as a refinement of \cite[Theorem~1.1]{KS20}, and present several methods for computing $d(\Sigma(p,q,r))$ in the case where $p$ is odd and $pq+pr-qr=1$ (Subsection~\ref{subsec(Some-properties)}). 
In Section~\ref{sec(A-parity-independent-inequality)}, we establish the inequality in Proposition~\ref{prop(D-ineq)} relating the Ozsv\'ath-Szab\'o $d$-invariants for certain classes of Brieskorn homology $3$-spheres with almost simple linear graphs. 
In Section~\ref{sec(infinite-kinds-of-infinite-classes)}, we present explicit computations of $d(\Sigma(p,q,r))$ for infinitely many families. 
In Section~\ref{sec(d-neq-D)}, we show that there exist infinitely many families of $\Sigma(p,q,r)$ that do not satisfy the inequality established in Section~\ref{sec(A-parity-independent-inequality)}. 
In Section~\ref{sec(Fibonacci-cases)}, we investigate the Fibonacci cases $d(\Sigma(F_{2k+1},F_{2k+2},F_{2k+3}))$. 
Finally, in Section~\ref{sec(application-knot-conc)}, we present applications of the results of Sections~\ref{sec(infinite-kinds-of-infinite-classes)} and~\ref{sec(d-neq-D)} to the knot concordance group. 

In this paper, unless otherwise specified, we assume that all manifolds are smooth, connected, and oriented, and that all maps are smooth. 

\section*{Acknowledgements}
I would like to thank my adviser when I was a graduate student, Hisaaki Endo, for encouraging me to write this paper. 
I would also like to express my sincere gratitude to Motoo Tange for contributing to his knowledge of Brieskorn homology spheres and the $d$-invariants. 
I would like to thank Naoki Kuroda for teaching me the contents of Remark \ref{rmk(Kuroda)}. 
I also would like to thank Kazuhiro Ichihara, Nobuo Iida, Hokuto Konnno, Ryotaro Kosuge, Taketo Sano, and Akihiro Takano for helpful comments. 
I would like to thank O\u{g}uz \c{S}avk for his valuable advice on the background and direction of this paper, as well as for his helpful comments and corrections on a previous version of this paper. 

\section{Preliminaries}
\label{sec(Preliminaries)}
\subsection[Brieskorn homology $3$-spheres]{Brieskorn homology  \texorpdfstring{$3$}{3}-spheres}
\label{subsec(BrieskornS3)}
Let $p, q$, and $r$ be integers with $1 \le p < q < r$ such that 
$\gcd(p,q)=\gcd(q,r)=\gcd(r,p)=1$. 
The symbol $S_{\varepsilon}^{5}$ denotes the $5$-sphere of radius $\varepsilon$ centered at the origin, where $0<\varepsilon\ll1$. 

\begin{defi}[Brieskorn homology $3$-sphere]
The topological $3$-manifold
\[
\Sigma(p,q,r)
:=
\{(x,y,z)\in \mathbb{C}^{3}\mid x^{p}+y^{q}+z^{r}=0\}
\cap S_{\varepsilon}^{5}
\]
is called the \textit{Brieskorn homology $3$-sphere}.
\end{defi}

\begin{rmk}
Let $p_i,q_i$, and $r_i$ be integers with $1\le p_i<q_i<r_i$ such that $\gcd(p_i,q_i)=\gcd(q_i,r_i)=\gcd(r_i,p_i)=1 \quad(i=1,2)$. 
If $(p_1,q_1,r_1) \neq (p_2,q_2,r_2)$, then 
$\pi_1(\Sigma(p_1,q_1,r_1))$ is not isomorphic to 
$\pi_1(\Sigma(p_2,q_2,r_2))$. 
Hence $\Sigma(p_1,q_1,r_1)$ is not homotopy equivalent to $\Sigma(p_2,q_2,r_2)$. 
\end{rmk}

\begin{rmk}
Fix $\delta$ satisfying $0<\delta\ll\varepsilon\ll1$. 
The symbol $D_{\varepsilon}^{6}$ denotes the $6$-ball of radius $\varepsilon$ centered at the origin. 
The Milnor fiber of the Brieskorn singularity $x^{p}+y^{q}+z^{r}=0$ is defined by
\[
M(p,q,r)
:=
\{(x,y,z)\in\mathbb{C}^{3}\mid x^{p}+y^{q}+z^{r}=\delta\}
\cap D_{\varepsilon}^{6}.
\]
This is a smooth $4$-manifold whose boundary is naturally identified with the Brieskorn homology $3$-sphere $\Sigma(p,q,r)$. 
For this reason, we regard $\Sigma(p,q,r)$ as a smooth manifold.
\end{rmk}

\begin{rmk}
Since $\Sigma(p,q,r)$ is a homology $3$-sphere, there exist integers 
$e_0,p',q',r'$ satisfying the Diophantine equation
\[
e_0pqr+p'qr+pq'r+pqr'=-1,
\qquad
0<p'<p,\;
0<q'<q,\;
0<r'<r .
\]

Let $p_1,p_2$, and $p_3$ denote $p,q$, and $r$, and $p'_1,p'_2$, and $p'_3$ denote $p',q'$, and $r'$, respectively. 
Then each rational number $-p_i/p'_i$ admits a unique continued fraction expansion of the form 
\[
-\frac{p_i}{p'_i}
=
t_{i1}-\cfrac{1}{t_{i2}-\cfrac{1}{\ddots-\cfrac{1}{t_{im_i}}}},
\qquad
t_{im_i}\neq -1,
\quad (i=1,2,3).
\]

Let $t$ be an integer and let $s$ be a rational number.
There is a move on surgery diagrams called the {\it slam-dunk} move, which preserves the diffeomorphism type of the resulting closed
$3$-manifold (see Figure~\ref{fig(slam-dunk)}). 

\begin{figure}[htbp]
    \centering
\begin{overpic}[scale=0.5
]{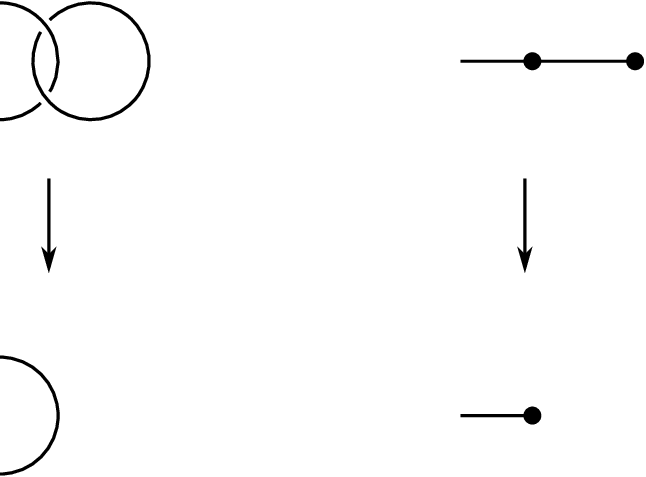}
    \put(3,75){$t$}
    \put(15,75){$s$}
    \put(-10,71.5){$\cdots$}
    \put(-10,53){$\cdots$}
    \put(30,36){slam-dunk}
    \put(11,10){$t-\displaystyle\frac{1}{s}$}
    \put(-10,16.5){$\cdots$}
    \put(-10,-1.5){$\cdots$}

    \put(45,62){$=$}
    \put(45,8){$=$}

    \put(82,68){$t$}
    \put(98,68){$s$}
    \put(60,62.25){$\cdots$}
    \put(87,9){$t-\displaystyle\frac{1}{s}$}
    \put(60,7.5){$\cdots$}
\end{overpic}
\caption{Slam-dunk}
\label{fig(slam-dunk)}
\end{figure}

Applying this move repeatedly, the surgery diagram of $\Sigma(p,q,r)$
can be converted into the plumbing graph shown in
Figure~\ref{fig(BrieskornS3)}. 

\begin{figure}[htbp]
    \centering
\begin{overpic}[scale=0.5]{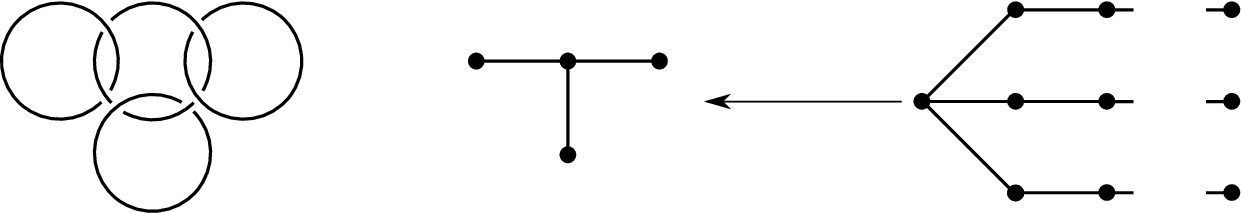}
    \put(11,18){$e_0$}
    \put(0.5,1){$-\displaystyle\frac{p}{p'}$}
    \put(-7,14){$-\displaystyle\frac{q}{q'}$}
    \put(25,14){$-\displaystyle\frac{r}{r'}$}

    \put(33,8){$=$}

    \put(44.5,14){$e_0$}
    \put(38,3){$-\displaystyle\frac{p}{p'}$}
    \put(35,16.5){$-\displaystyle\frac{q}{q'}$}
    \put(50,16.5){$-\displaystyle\frac{r}{r'}$}

    \put(57.5,11){slam-dunk}

    \put(72.5,5){$e_0$}
    \put(82,3.5){$t_{31}$}
    \put(82,10.75){$t_{21}$}
    \put(82,18){$t_{11}$}
    \put(89.5,3.5){$t_{32}$}
    \put(89.5,10.75){$t_{22}$}
    \put(89.5,18){$t_{12}$}
    \put(92.5,0.75){$\cdots$}
    \put(92.5,8.1){$\cdots$}
    \put(92.5,15.5){$\cdots$}
    \put(99,3.5){$t_{3m_{3}}$}
    \put(99,10.75){$t_{2m_{2}}$}
    \put(99,18){$t_{1m_{1}}$}
\end{overpic}
\caption{A surgery diagram of $\Sigma(p,q,r)$}
\label{fig(BrieskornS3)}
\end{figure}

In Figure~\ref{fig(BrieskornS3)}, the right side of the diagram corresponds to a surgery diagram for a $3$-manifold obtained by performing resolution of $\Sigma(p,q,r)$, which resolves singularities in $\Sigma(p,q,r)$ separately from the method using Milnor fiber.

Note that $\Sigma(1,q,r)=S^{3}$ and $\Sigma(2,3,5)$ is diffeomorphic to the Poincar\'e homology sphere. 
The Poincar\'e homology sphere is the first known example of a homology $3$-sphere that is not homeomorphic to the $3$-sphere $S^{3}$. 
\end{rmk}

\subsection[The $d$-invariants of Brieskorn homology $3$-spheres]{The \texorpdfstring{$d$}{d}-invariants of Brieskorn homology \texorpdfstring{$3$}{3}-spheres}
\label{subsec(d-inv)}

In this subsection we review some properties of the Ozsv\'ath-Szab\'o $d$-invariant and several previous results that will be used later. 

By \cite[Proposition 4.2]{OS03}, the $d$-invariant satisfies $d(-Y,\mathfrak{s})=-d(Y,\mathfrak{s})$. 
Throughout this paper we choose the orientation of $\Sigma(p,q,r)$ so that $d(\Sigma(p,q,r))\ge0$. 

If $Y$ is an integral homology $3$-sphere, then it admits a unique
$\text{spin}^c$ structure. Hence we simply write $d(Y)$ instead of
$d(Y,\mathfrak{s})$. Note that $d(Y)$ is an even integer.

If $\Sigma(p,q,r)$ is obtained by Dehn surgery on $S^3$ along a knot,
then its $d$-invariant can sometimes be computed systematically as follows.

Let
\[
g(p,q):=\frac{(p-1)(q-1)}{2}
\quad\text{and}\quad
S(p,q):=\{ap+bq \mid (a,b)\in\mathbb{Z}_{\ge0}^{2}\}.
\]

Let $T_{p,q}$ denote the right-handed $(p,q)$-torus knot, and $S^3_r(K)$ denote the $3$-manifold obtained from $S^3$ by Dehn surgery
along a knot $K$ with slope $r\in\mathbb{Q}\cup\{\infty\}$.
Then
\[
S^3_{-1/n}(-T_{p,q})=\Sigma(p,q,pqn-1)
\quad\text{and}\quad
S^3_{-1/n}(T_{p,q})=\Sigma(p,q,pqn+1)
\]
(see \cite[Example 1.2]{Sav02}).

\begin{thm}[Tweedy~\cite{Twe13}]
The following hold:
\[
d(\Sigma(p,q,pqn-1))
=
2|\{s\notin S(p,q)\mid s\ge g(p,q)\}|
\quad\text{and}\quad
d(\Sigma(p,q,pqn+1))=0 .
\]
\end{thm}

Next we review results from \cite{KS20}, which will be used in
Section~\ref{sec(refinement)}. For details we refer the reader to
\cite{KS20}.

We consider the Brieskorn homology $3$-sphere $\Sigma(p,q,r)$ satisfying 
\[
pq+pr-qr=1.
\]

If $pq+pr-qr=1$, then a surgery diagram of $\Sigma(p,q,r)$ is depicted in Figure~\ref{BrieskornS3(pq+pr-qr=1)}.

\begin{figure}[htbp]
\vspace{5truemm}
\centering
\begin{overpic}[scale=0.5]{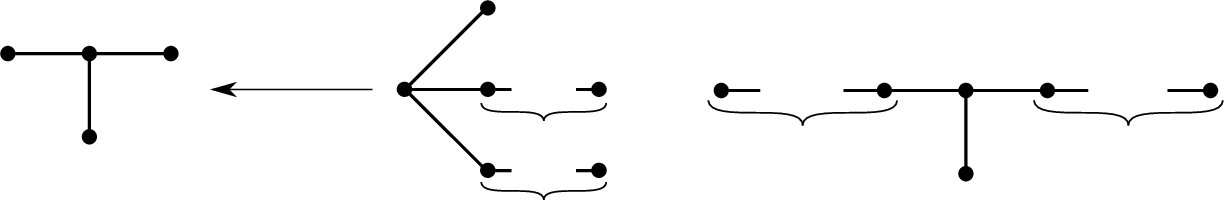}
    \put(5.5,13.5){$-2$}
    \put(0.5,2){$-\displaystyle\frac{p}{1}$}
    \put(-9,16){$-\displaystyle\frac{q}{q-1}$}
    \put(11,15.5){$-\displaystyle\frac{r}{r-1}$}

    \put(16,5){slam-dunk}

    \put(29,11){$-2$}
    \put(33.5,0){$-2$}
    \put(38,11){$-2$}
    \put(38,17){$-p$}

    \put(42,1.5){$\cdots$}
    \put(42,8){$\cdots$}

    \put(41,-3){$r-1$}
    \put(41,4){$q-1$}

    \put(50,1){$-2$}
    \put(47,11){$-2$}

    \put(53.5,8){$=$}

    \put(57,11){$-2$}
    \put(70,11){$-2$}
    \put(76,11){$-2$}
    \put(82,11){$-2$}
    \put(95,11){$-2$}

    \put(63.25,8){$\cdots$}
    \put(89.75,8){$\cdots$}

    \put(61.5,3){$q-1$}
    \put(88.5,3){$r-1$}

    \put(76,-1){$-p$}
\end{overpic}
\vspace{5truemm}
\caption{A surgery diagram of $\Sigma(p,q,r)$ with $pq+pr-qr=1$}
\label{BrieskornS3(pq+pr-qr=1)}
\end{figure}

Suppose that $p$ is odd. 
Then we can write $p=2n_p+1$ for some integer $n_p$. 

Define
\begin{align*}
\mathfrak{L}_{p}
&:=
\{(a,m)\in(2\mathbb{Z}+1)\times\mathbb{Z}
\mid
-p\le a\le p,\;
0\le m\le n_p\} \quad\text{and}\\
R_{p,q}
&:=
\{(a,m)\in\mathfrak{L}_{p}\mid
F_{p,q}(a,m)\ge F_{p,q}(1,1)\}.
\end{align*}

\begin{thm}[Karakurt-\c{S}avk~{\cite[Theorem~1.1]{KS20}}]
\label{thm([KS20]main)}
If $p$ is odd and $pq+pr-qr=1$, then
\[
d(\Sigma(p,q,r))
=
\max_{(a,m)\in R_{p,q}}F_{p,q}(a,m).
\]
\end{thm}

In this paper we refer to the formula in
Theorem~\ref{thm([KS20]main)} for computing the $d$-invariant as the {\it Karakurt-\c{S}avk formula}. 

Recall that $l_{p,q}=q-p$. 
Let $t_{p,q}$ and $\alpha_{p,q}$ denote the quotient
and remainder obtained when dividing $n_p$ by $l_{p,q}$, respectively. 

Recall that
\[
D(p,q,r) =
\begin{cases}
(t_{p,q}+1)(n_p+\alpha_{p,q}) & \text{if $p$ is odd},\\
d(\Sigma(p,q,r)) & \text{if $p$ is even}.
\end{cases}
\]

From the Karakurt-\c{S}avk formula we obtain the following results.

\begin{thm}[Karakurt-\c{S}avk~{\cite[Theorem~1.2]{KS20}}]
\label{thm(Example[KS20])}
The following hold.
\begin{itemize}
\item[(1)] $d(\Sigma(2n+1,4n+1,4n+3))=2n$ for any $n\ge1$.  \\
Equivalently, if $p$ is odd, $pq+pr-qr=1$, and $l_{p,q}=2n_p$, then $d(\Sigma(p,q,r))=D(p,q,r)$. \smallskip

\item[(2)] $d(\Sigma(2n+1,3n+2,6n+1))=2n$ for any $n\ge1$.  \\
Equivalently, if $p$ is odd, $pq+pr-qr=1$, and $l_{p,q}=n_p+1$, then $d(\Sigma(p,q,r))=D(p,q,r)$. \smallskip

\item[(3)] $d(\Sigma(2n+1,3n+1,6n+5))=2n$ for any $n\ge1$.  \\
Equivalently, if $p$ is odd, $pq+pr-qr=1$, and $l_{p,q}=n_p$, then $d(\Sigma(p,q,r))=D(p,q,r)$. \smallskip

\item[(4)] $d(\Sigma(4n+3,5n+4,20n+11))=6n+2$ for any $n\ge1$.  \\
Equivalently, if $p$ is odd, $pq+pr-qr=1$, and $l_{p,q}=n_p/2+1$, then $d(\Sigma(p,q,r))=D(p,q,r)$. \smallskip

\item[(5)] $d(\Sigma(2n+1,2n+2,4n^2+6n+1))=n^2+n$ for any $n\ge1$.  \\
Equivalently, if $p$ is odd, $pq+pr-qr=1$, and $l_{p,q}=1$, then
$d(\Sigma(p,q,r))=D(p,q,r)$. 
\end{itemize}
\end{thm}

We also define $\Delta_{p,q}(m)
:=
4(2m-(q+r))^2-16(q+r)(p-1)$,
\[
\mathfrak{t}_{p,q}(m)
:=
\frac{\sqrt{\Delta_{p,q}(m)}}{2(q+r)},
\quad
\mathfrak{c}_{p,q}(m)
:=
\frac{2qm}{q+r},
\quad\text{and}\quad
\mathfrak{d}_{p,q}(m)
:=
\min_{a\in2\mathbb{Z}+1}
|\mathfrak{c}_{p,q}(m)-a|.
\]

There is a necessary and sufficient condition for a pair $(a,m)$
to lie in $R_{p,q}$.

\begin{prop}[Karakurt-\c{S}avk~{\cite[Proposition 4.8]{KS20}}]
\label{prop(EquivCond)}
A pair $(a,m)$ belongs to $R_{p,q}$ if and only if either $(a,m)=(1,1)$
or all of the following conditions hold:
\begin{enumerate}
\item $m\ge2$,
\item $\Delta_{p,q}(m)\ge0$,
\item $\mathfrak{d}_{p,q}(m)\le\mathfrak{t}_{p,q}(m)$,
\item $|\mathfrak{c}_{p,q}(m)-a|=\mathfrak{d}_{p,q}(m)$.
\end{enumerate}
\end{prop}

\subsection{The homology cobordism group}
\label{subsec(homology-cobordism-group)}

In this subsection we review several results for the $3$-dimensional homology cobordism group $\Theta_{\mathbb{Z}}^3$ that will be related to later sections. 
For further details we refer the reader to the survey~\cite{Sav24}.

First we recall the notions of $h$-cobordism and the homotopy cobordism group.

\begin{defi}[$h$-cobordism]
Let $M_0$ and $M_1$ be homotopy $n$-spheres. 
If there exists an $(n+1)$-manifold $W$ such that
\begin{itemize}
\item $\partial W = -M_0 \cup M_1$, and
\item the inclusions $M_0 \hookrightarrow W$ and $M_1 \hookrightarrow W$ are homotopy equivalences,
\end{itemize}
then $M_0$ and $M_1$ are said to be \textit{$h$-cobordant}. 
We write $M_0 \sim M_1$ in this case.
\end{defi}

Note that $\sim$ is an equivalence relation on the set of homotopy $n$-spheres. 
The symbol $[M]_{\sim}$ denotes the equivalence class of a homotopy $n$-sphere $M$ under $\sim$.

\begin{defi}[The $n$-dimensional homotopy cobordism group]
Let $\Theta^n$ denote the set of $h$-cobordism classes of homotopy $n$-spheres. 
The connected sum operation induces an abelian group structure 
\[
[M_0]_{\sim} + [M_1]_{\sim} := [M_0 \# M_1]_{\sim}.
\]
The identity element is the class $[S^n]_{\sim}$ of the standard $n$-sphere $S^n$, 
and the inverse of $[M]_{\sim}$ is $[-M]_{\sim}$.  
The group $\Theta^n$ is called the \textit{$n$-dimensional homotopy cobordism group}.
\end{defi}

\begin{defi}[Homology cobordism]
Let $M_0$ and $M_1$ be homology $n$-spheres. 
If there exists an $(n+1)$-manifold $W$ such that
\begin{itemize}
\item $\partial W = -M_0 \cup M_1$, and
\item the inclusions $M_0 \hookrightarrow W$ and $M_1 \hookrightarrow W$ induce isomorphisms 
\[
H_k(M_0;\mathbb{Z}) \cong H_k(W;\mathbb{Z}) \cong H_k(M_1;\mathbb{Z})
\]
for all $k$,
\end{itemize}
then $M_0$ and $M_1$ are said to be \textit{homology cobordant}. 
We write $M_0 \sim_{\mathbb{Z}} M_1$ in this case. 
\end{defi}

Note that $\sim_{\mathbb{Z}}$ is an equivalence relation on the set of homology $n$-spheres. 
The symbol $[M]_{\sim_{\mathbb{Z}}}$ denotes the equivalence class of a homology $n$-sphere $M$ under $\sim_{\mathbb{Z}}$. 

\begin{rmk}
\label{rmk(notcobordant)}
The $d$-invariant is invariant under homology cobordism for homology $3$-spheres. 
Hence, if $Y_1$ and $Y_2$ are homology $3$-spheres and $d(Y_1) \neq d(Y_2)$, then $Y_1$ and $Y_2$ are not homology cobordant. 
\end{rmk}

Throughout this paper we assume that a homology $3$-sphere $Y$ is oriented so that $d(Y) \ge 0$. 

\begin{defi}[The $n$-dimensional homology cobordism group]
Let $\Theta_{\mathbb{Z}}^n$ denote the set of homology cobordism classes of homology $n$-spheres. 
The connected sum operation induces an abelian group structure
\[
[M_0]_{\sim_{\mathbb{Z}}} + [M_1]_{\sim_{\mathbb{Z}}}
:= [M_0 \# M_1]_{\sim_{\mathbb{Z}}}.
\]
The identity element is the class $[S^n]_{\sim_{\mathbb{Z}}}$  of the standard $n$-sphere $S^n$, 
and the inverse of $[M]_{\sim_{\mathbb{Z}}}$ is $[-M]_{\sim_{\mathbb{Z}}}$. 
The group $\Theta_{\mathbb{Z}}^n$ is called the \textit{$n$-dimensional homology cobordism group}.
\end{defi}

When $n \neq 3$, it is known that $\Theta^n$ is finite by \cite[Theorem 1.2]{KM63}, and that $\Theta^n$ is isomorphic to $\Theta_{\mathbb{Z}}^n$ by \cite[Theorem I.2]{GAn70}. 

The $3$-dimensional homotopy cobordism group $\Theta^3$ is trivial as a consequence of Perelman's resolution of the Poincaré conjecture \cite{Per02}, \cite{Per03a}, \cite{Per03b}. 

In contrast, the $3$-dimensional homology cobordism group $\Theta_{\mathbb{Z}}^3$ is known to be nontrivial.

\begin{thm}[Rokhlin~\cite{Rok52}]
There exists a surjective homomorphism
\[
\mu : \Theta_{\mathbb{Z}}^3 \rightarrow \mathbb{Z}_2 .
\]
\end{thm}

Later it was shown that $\Theta_{\mathbb{Z}}^3$ is infinite.

\begin{thm}[Fintushel-Stern~\cite{FS85}]
The group $\Theta_{\mathbb{Z}}^3$ contains the subgroup
\[
\mathbb{Z}[\Sigma(2,3,5)]_{\sim_{\mathbb{Z}}}
=
\mathbb{Z}[S^3_{-1}(-T_{2,3})]_{\sim_{\mathbb{Z}}}.
\]
\end{thm}

Furthermore $\Theta_{\mathbb{Z}}^3$ contains a $\mathbb{Z}^{\infty}$-subgroup.

\begin{thm}[Furuta~\cite{Fur90}]
\label{thm(Fur90)}
The group $\Theta_{\mathbb{Z}}^3$ contains the subgroup 
\[
\bigoplus_{n=1}^{\infty}
\mathbb{Z}[\Sigma(2,3,6n-1)]_{\sim_{\mathbb{Z}}}
=
\bigoplus_{n=1}^{\infty}
\mathbb{Z}[S^3_{-1/n}(-T_{2,3})]_{\sim_{\mathbb{Z}}}.
\]
\end{thm}

Let $(p_n,q_n,r_n)\in\mathbb{Z}^3$ satisfy $1<p_n<q_n<r_n$,
$\gcd(p_n,q_n)=\gcd(q_n,r_n)=\gcd(r_n,p_n)=1$, and $p_n q_n + p_n r_n - q_n r_n = 1$. 
In this case the Fintushel-Stern invariant $R(p_n,q_n,r_n)$ is equal to $1$ (see \cite{NZ85}). 
Since this value is positive, the following statement follows from the proof of Theorem~\ref{thm(Fur90)}.

\begin{thm}[Furuta~\cite{Fur90}]
\label{thm(Fur90reinterpretation)}
Let $p_n, q_n$, and $r_n$ be positive integers. 
If
\[
p_nq_n+p_nr_n-q_nr_n=1
 \quad \text{and} \quad 
p_nq_nr_n<p_{n+1}q_{n+1}r_{n+1}, 
\]
then $\Theta_{\mathbb{Z}}^3$ contains the subgroup
\[
\bigoplus_{n=1}^{\infty}
\mathbb{Z}[\Sigma(p_n,q_n,r_n)] .
\]
\end{thm}

On the other hand, it is known that $\Theta_{\mathbb{Z}}^3$ has a $\mathbb{Z}$-summand.

\begin{thm}[Fr{\o}yshov~\cite{Fro02}]
There exists a subgroup $A$ of $\Theta_{\mathbb{Z}}^3$ such that
\[
\Theta_{\mathbb{Z}}^3
=
A \oplus
\mathbb{Z}[\Sigma(2,3,5)]_{\sim_{\mathbb{Z}}}.
\]
\end{thm}

Furthermore, the group $\Theta_{\mathbb{Z}}^3$ has a $\mathbb{Z}^{\infty}$-summand.

\begin{thm}[Dai-Hom-Stoffregen-Truong~\cite{DHST23}]
There exists a subgroup $A$ of $\Theta_{\mathbb{Z}}^3$ such that
\[
\Theta_{\mathbb{Z}}^3
=
A \oplus
\bigoplus_{n=1}^{\infty}
\mathbb{Z}[\Sigma(2n+1,4n+1,4n+3)]_{\sim_{\mathbb{Z}}}.
\]
\end{thm}

Later, a simpler method was introduced in \cite{KS22}, which shows that 
$\Theta_{\mathbb{Z}}^3$ admits several types of $\mathbb{Z}^{\infty}$-summands.

\begin{thm}[Karakurt-\c{S}avk~\cite{KS22}]
\label{thm([KS22])}
There exists a subgroup $A$ of $\Theta_{\mathbb{Z}}^3$ such that 
\begin{align*}
\Theta_{\mathbb{Z}}^3
=A&\oplus\bigoplus_{n=1}^{\infty}\mathbb{Z}[\Sigma(2n+1,4n+1,4n+3)]_{\sim_{\mathbb{Z}}}\\
&\oplus\bigoplus_{n=2}^{\infty}\mathbb{Z}[\Sigma(2n+1,3n+2,6n+1)]_{\sim_{\mathbb{Z}}}
\oplus\bigoplus_{n=1}^{\infty}\mathbb{Z}[\Sigma(2n+1,3n+1,6n+5)]_{\sim_{\mathbb{Z}}}.
\end{align*}
\end{thm}

\begin{rmk}
The Brieskorn homology $3$-spheres
\[
\Sigma(2n+1,4n+1,4n+3), \quad
\Sigma(2n+1,3n+2,6n+1), \quad\text{and}\quad
\Sigma(2n+1,3n+1,6n+5)
\]
have almost simple linear plumbing graphs for all $n\ge1$. 
Moreover, by \cite[Theorem~\ref{thm(Example[KS20])}]{KS20},
\[
d(\Sigma(2n+1,4n+1,4n+3))
=
d(\Sigma(2n+1,3n+2,6n+1))
=
d(\Sigma(2n+1,3n+1,6n+5))
=2n .
\]
\end{rmk}

The following open problem remains.

\begin{prob}
Is $\Theta_{\mathbb{Z}}^3$ isomorphic to $\mathbb{Z}^{\infty}$?
\end{prob}

At present, the only known candidate for torsion in $\Theta_{\mathbb{Z}}^3$ is of order $2$, proposed in \cite{BC24}.

\subsection{The knot concordance group}
\label{subsec(knot-concordance-group)}

Let $K_0$ and $K_1$ be oriented knots in the $3$-sphere $S^3$, and $K_0 \# K_1$ denote their connected sum. 
We first recall the definitions of knot concordance and the knot concordance group. 

\begin{defi}[Knot concordance]
Two oriented knots $K_0$ and $K_1$ are said to be \textit{concordant} if there exists a smoothly embedded annulus $C$ in $S^3 \times [0,1]$ such that 
\begin{itemize}
\item $C$ is diffeomorphic to $S^1 \times [0,1]$, and
\item $\partial C = (-K_0)\times\{0\} \cup K_1 \times\{1\}$. 
\end{itemize}
In this case, we write $K_0 \sim_c K_1$. 
\end{defi}

It is well known that $\sim_c$ is an equivalence relation on the set of oriented knots in $S^3$. 
Let $[K]_{\sim_c}$ denote the equivalence class of a knot $K$ under $\sim_c$.

\begin{defi}[The knot concordance group]
The quotient set
\[
\mathcal{C}
:=\{\text{oriented knots in } S^3\}/\sim_c
\]
forms a commutative group under the operation
\[
[K_0]_{\sim_c} + [K_1]_{\sim_c}
:= [K_0 \# K_1]_{\sim_c}.
\]
The identity element is the class $[O]_{\sim_c}$ of the unknot $O$, 
and the inverse of $[K]_{\sim_c}$ is $[-K]_{\sim_c}$.
This group is called the \textit{knot concordance group}.
\end{defi}

Let $\mathcal{C}_{\mathrm{TS}}$ denote the subgroup of $\mathcal{C}$ generated by topologically slice knots.

Suppose $p,q$, and $r$ are pairwise relatively prime positive integers with $1<p<q<r$. 
Then the double branched cover of $S^3$ branched along the pretzel knot $P(-p,q,r)$ is the Brieskorn homology sphere $\Sigma(p,q,r)$. 

Using gauge theoretic techniques developed in~\cite{Fur90}, the following was shown. 

\begin{cor}[Endo~{\cite[Corollary 3]{End95}}]
\label{cor(End95)}
Let $K_n$ be one of the following pretzel knots: 
\[
\begin{aligned}
&P(-2n-1,4n+1,4n+3), \ P(-2n-1,2n+3,2n^2+4n+1), \\
&P(-2n-1,2n+5,n^2+3n+1), \ P(-4n-1,6n+1,12n+5), \\
&P(-4n-3,6n+5,12n+7). 
\end{aligned}
\]
Then the subgroup $\mathcal{C}_{\mathrm{TS}}$ contains a subgroup 
\[
\bigoplus_{n=1}^{\infty}\mathbb{Z}[K_n].
\]
\end{cor}

\begin{rmk}
All the examples $P(-p,q,r)$ appearing in Corollary~\ref{cor(End95)} satisfy $pq+pr-qr=1$. 
\end{rmk}

Let $n$ be a positive integer and $(p_n,q_n,r_n)\in\mathbb{Z}^3$ satisfy
\[
1<p_n<q_n<r_n,\quad
\gcd(p_n,q_n)=\gcd(q_n,r_n)=\gcd(r_n,p_n)=1,
\]
and
\[
p_nq_n+p_nr_n-q_nr_n=1.
\]

When $p_n,q_n$, and $r_n$ are all odd, the pretzel knot $P(-p_n,q_n,r_n)$ has Alexander polynomial
\[
\Delta_{P(-p_n,q_n,r_n)}(t)
=
\frac{(-p_nq_n+q_nr_n-r_np_n)(t-1)^2+(t+1)^2}{4t}
=1.
\]

In this case the Fintushel-Stern invariant $R(p_n,q_n,r_n)$ is equal to $1$.
Since this value is positive, the following statement follows from the proof of Corollary~\ref{cor(End95)} in \cite{End95}.

\begin{thm}[Endo~\cite{End95}]
\label{thm([End95]reinterpretation)}
Let $p_n,q_n$, and $r_n$ be odd positive integers. 
If
\[
p_nq_n+p_nr_n-q_nr_n=1
 \quad \text{and} \quad 
p_nq_nr_n<p_{n+1}q_{n+1}r_{n+1}, 
\]
then $\mathcal{C}_{\mathrm{TS}}$ contains a subgroup 
\[
\bigoplus_{n=1}^{\infty}\mathbb{Z}[P(-p_n,q_n,r_n)].
\]
\end{thm}

Using the construction of $\mathbb{Z}^{\infty}$-summands in $\Theta_{\mathbb{Z}}^3$
given in Theorem~\ref{thm([KS22])}, it was further shown that some infinite families of
pretzel knots appearing in Corollary~\ref{cor(End95)} actually generate
$\mathbb{Z}^{\infty}$-summands of $\mathcal{C}_{\mathrm{TS}}$.

\begin{cor}[Karakurt-\c{S}avk~{\cite[Corollary 1.5]{KS22}}]
There exist submodules $A_1$ and $A_2$ of $\mathcal{C}_{\mathrm{TS}}$ such that 
\[
\mathcal{C}_{\mathrm{TS}}
=
A_1 \oplus 
\bigoplus_{n=1}^{\infty}\mathbb{Z}[P(-4n-1,6n+1,12n+5)] 
=
A_2 \oplus 
\bigoplus_{n=2}^{\infty}\mathbb{Z}[P(-4n-3,6n+5,12n+7)].
\]
\end{cor}

\section{Refinement of the Karakurt-\c{S}avk formula and some properties}
\label{sec(refinement)}
\subsection{Refinement of the Karakurt-\c{S}avk formula} 
Throughout the subsequent sections, we assume that $p$, $q$, and $r$ are integers satisfying $1 < p < q < r$ and
\[
pq + pr - qr = 1.
\]
It follows that $\gcd(p,q)=\gcd(q,r)=\gcd(r,p)=1$. 
In order to compute the $d$-invariant of any Brieskorn homology $3$-sphere $\Sigma(p,q,r)$ with $p$ odd and $pq + pr - qr = 1$, we first establish several lemmas that will be needed later. 

\begin{lem}
\label{lem(inequality[q])}
If $p$ is odd and $pq + pr - qr = 1$, then $p + 1 \le q \le 2p - 1$. 
\end{lem}

\begin{proof}
Note that
\[
3 \le p < q < r = \frac{pq - 1}{q - p}.
\]
It follows that $q^2 - 2pq + 1 < 0$, 
which implies $p - \sqrt{p^2 - 1} < q < p + \sqrt{p^2 - 1}$. 
Since $q > p > p - \sqrt{p^2 - 1}$ and $2p - 1 < p + \sqrt{p^2 - 1} < 2p$, we obtain the desired inequality.
\end{proof}

Define the two-variable quadratic functions
\[
f_{p,q}(x,y) := -(q+r)x^2 + 4qxy - 4(q-p)y^2 - 4y, \quad 
F_{p,q}(x,y) := \frac{f_{p,q}(x,y) + q + r}{4}.
\]

\begin{lem}
\label{lem(F(a,m))}
Suppose that $p$ is odd and $pq + pr - qr = 1$. 
For any complex numbers $a$ and $m$, the function $F_{p,q}(a,m)$ can be written as
\[
F_{p,q}(a,m)
= -\frac{(2(q-p)m - aq + 1)^2 - (q-a)^2}{4(q-p)}.
\]
\end{lem}

\begin{proof}
This follows from straightforward calculations:
\begin{align*}
F_{p,q}(a,m) 
&= \frac{-(q+r)a^2 + 4aqm - 4(q-p)m^2 - 4m + q + r}{4} \\
&= \frac{-4(q-p)m^2 + 4(aq-1)m - (q+r)(a^2-1)}{4} \\
&= -(q-p)m^2 + (aq-1)m - \frac{(a^2-1)(q^2-1)}{4(q-p)}\\
&= -(q-p)\left(m - \frac{aq-1}{2(q-p)}\right)^2 
   + (q-p)\left(\frac{aq-1}{2(q-p)}\right)^2 
   - \frac{(a^2-1)(q^2-1)}{4(q-p)} \\
&= -(q-p)\left(m - \frac{aq-1}{2(q-p)}\right)^2 
   + \frac{(aq-1)^2 - (a^2-1)(q^2-1)}{4(q-p)} \\
&= -\frac{(2(q-p)m - aq + 1)^2}{4(q-p)} 
   + \frac{(q - a)^2}{4(q-p)} \\
&= -\frac{(2(q-p)m - aq + 1)^2 - (q-a)^2}{4(q-p)}.
\end{align*}
This completes the proof.
\end{proof}

\begin{lem}
\label{lem(F(1,m))}
Suppose that $p$ is odd and $pq + pr - qr = 1$. 
For any complex number $m$, the function $F_{p,q}(1,m)$ is given by
\[
F_{p,q}(1,m) = -(q-p)m^2 + (q-1)m.
\]
In particular, $F_{p,q}(1,1) = p-1$. 
\end{lem}

\begin{proof}
From the proof of Lemma \ref{lem(F(a,m))}, we obtain
\begin{align*}
F_{p,q}(1,m)
&= -(q-p)m^2 + (q-1)m - \frac{(1-1)(q^2-1)}{4(q-p)} \\
&= -(q-p)m^2 + (q-1)m.
\end{align*}
This completes the proof.
\end{proof}

We write $n_p$ in the form
\[
n_p = (q-p)t_{p,q} + \alpha_{p,q},
\]
where $t_{p,q}$ and $\alpha_{p,q}$ are the quotient and remainder, respectively. 

Recall that
\begin{align*}
\mathfrak{L}_{p} &= \{(a,m) \in (2\mathbb{Z}+1) \times \mathbb{Z} \mid -p \le a \le p, \; 0 \le m \le n_p\}, \quad\text{and}\\
R_{p,q} &= \{(a,m) \in \mathfrak{L}_{p} \mid F_{p,q}(a,m) \ge F_{p,q}(1,1)\}.
\end{align*}

We also define
\[
\mathfrak{L}'_{p} := \{(a,m) \in (2\mathbb{Z}+1) \times \mathbb{Z} \mid 1 \le a \le p, \; 1 \le m \le n_p\}.
\]

\begin{lem}
\label{lem(L-to-L')}
Suppose that $p$ is odd and $pq + pr - qr = 1$. 
Then 
\[
\mathfrak{L}_p \cap R_{p,q} = \mathfrak{L}'_{p} \cap R_{p,q}.
\]
\end{lem}

\begin{proof}
Note that 
\[
\mathfrak{L}_{p} \setminus \mathfrak{L}'_{p} = \{(a,m) \in (2\mathbb{Z}+1) \times \mathbb{Z} \mid -p \le a \le -1 \text{ or } m = 0\}.
\]

By Proposition \ref{prop(EquivCond)}, $(a,0) \notin R_{p,q}$ for any 
$a \in \{\pm 1, \pm 3, \ldots, \pm p\}$.  

Since
\[
\mathfrak{c}_{p,q}(m) = \frac{2qm}{q+r} > 0
\]
for all $m \in \{1,2,\ldots,n_p\}$, we have
\[
|\mathfrak{c}_{p,q}(m) - a| > |\mathfrak{c}_{p,q}(m) - 1| 
\ge \min_{a \in 2\mathbb{Z}+1} |\mathfrak{c}_{p,q}(m) - a| = \mathfrak{d}_{p,q}(m)
\]
for any $a \in \{-1, -3, \ldots, -p\}$ and $m \in \{1,2,\ldots,n_p\}$.  

Then, by condition (4) in Proposition \ref{prop(EquivCond)}, it follows that $(a,m) \notin R_{p,q}$.  

Therefore, the desired equality holds.
\end{proof}

\begin{lem}
\label{lem(m-less-than-a)}
Suppose that $p$ is odd and $pq + pr - qr = 1$. 
Let $(a,m)$ be any element of $\mathfrak{L}'_p \cap R_{p,q}$. 
If $m \le a$, then
\[
F_{p,q}(a,m) \le F_{p,q}(1,1).
\]
\end{lem}

\begin{proof}
Assume that $m \le a$. Then we have
\begin{align*}
aq - 2(q-p)m - 1 &\ge aq - 2(q-p)a - 1 
= a(2p-q) - 1 \\
&\ge 2p - q - 1 \\
&\ge 2p - (2p-1) - 1 = 0
\end{align*}
by Lemma \ref{lem(inequality[q])}. 

By Lemma \ref{lem(F(a,m))}, it follows that
\begin{align*}
\lefteqn{4(q-p)(F_{p,q}(1,1)-F_{p,q}(a,m)) = -4(q-p)F_{p,q}(a,m)+ 4(q-p)F_{p,q}(1,1)}\\
&= (2(q-p)m - aq + 1)^2 - (q-a)^2 -(2(q-p)-q+1)^2 + (q-1)^2\\
&= (aq - 2(q-p)m - 1)^2 - (q-a)^2 + 4(p-1)(q-p) \\
&\ge (aq - 2(q-p)a - 1)^2 - (q-a)^2 + 4(p-1)(q-p) \\
&= (aq - 2ap + 1)^2 - (q-a)^2 + 4(p-1)(q-p) \\
&= (a^2-1)q^2 - 4(a^2 - 1)pq + 4(a^2-1)p^2 - 4(a-1)p + 4(a-1)q - (a^2 - 1) \\
&= (a^2-1)(q^2 - 4pq + 4p^2 -1 ) + (a-1)(q-p)\\
&= (a-1)((a+1)((2p-q)^2-1) + (q-p)) \ge 0.
\end{align*}

The last inequality follows from Lemma \ref{lem(inequality[q])} and the fact that $a \ge 1$.  

Therefore, the desired inequality holds.
\end{proof}

We determine the maximum of $F_{p,q}(1,m)$ for $1 \le m \le n_p$. 

\begin{prop}
\label{prop(maxF(1,m))}
Suppose that $p$ is odd and $pq + pr - qr = 1$. 
Then 
\[
\max_{1 \le m \le n_p} F_{p,q}(1,m) = F_{p,q}(1, t_{p,q} + \min\{\alpha_{p,q}, 1\}) = (t_{p,q}+1)(n_p+\alpha_{p,q}).
\]
\end{prop}

\begin{proof}
Using Lemma~\ref{lem(F(1,m))}, we compute
\begin{align*}
\lefteqn{\max_{1 \le m \le n_p} F_{p,q}(1,m)}\\
&= \max_{1 \le m \le n_p} (-(q-p)m^2 + (q-1)m) \\
&= \max_{1 \le m \le n_p}
   \Biggl(-(q-p)\biggl(m - \frac{q-1}{2(q-p)}\biggr)^2
         + (q-p)\biggl(\frac{q-1}{2(q-p)}\biggr)^2\Biggr) \\
&= \max_{1 \le m \le n_p}
   \Biggl(-(q-p)\biggl(m - \biggl(t_{p,q}
        + \frac{\alpha_{p,q}}{q-p} + \frac{1}{2}\biggr)\biggr)^2
        + (q-p)\biggl(t_{p,q}
        + \frac{\alpha_{p,q}}{q-p} + \frac{1}{2}\biggr)^2\Biggr).
\end{align*}

Here we used the identity
\[
\frac{q-1}{2(q-p)}
 = \frac{p-1 + q - p}{2(q-p)}
 = \frac{n_p}{q-p} + \frac{1}{2}
 = t_{p,q} + \frac{\alpha_{p,q}}{q-p} + \frac{1}{2}.
\]

\medskip

\noindent
\textbf{Case 1:} $\alpha_{p,q} = 0$.  

In this case, we have $n_p = (q-p)t_{p,q}$ and
\[
\min_{m \in \mathbb{Z}} \left| m - \left(t_{p,q} + \frac{\alpha_{p,q}}{q-p} + \frac{1}{2} \right) \right| 
=
\min_{m \in \mathbb{Z}} \left| m - \left(t_{p,q} + \frac{1}{2} \right) \right| 
= \left| t_{p,q} - \left(t_{p,q} + \frac{1}{2} \right) \right| = \frac{1}{2}.
\]

Since $0 < t_{p,q} \le n_p$ and
\begin{align*}
F_{p,q}(1,1) - F_{p,q}(1,t_{p,q})
&= (q-p)t_{p,q}^2 - (q-1)t_{p,q} + p-1 \\
&= t_{p,q} n_p - (q-p+2n_p)t_{p,q} + 2 n_p \\
&= -n_p (t_{p,q}-1) \le 0,
\end{align*}
it follows that $(1,t_{p,q}) \in \mathfrak{L}_p \cap R_{p,q}$. Therefore,
\begin{align*}
\max_{1 \le m \le n_p} F_{p,q}(1,m)
&= F_{p,q}(1,t_{p,q}) \\
&= -(q-p)\left(t_{p,q}-\left(t_{p,q}+\frac{1}{2}\right)\right)^2
   + (q-p)\left(t_{p,q}+\frac{1}{2}\right)^2 \\
&= -\frac{q-p}{4} + (q-p)\left(t_{p,q}+\frac{1}{2}\right)^2 \\
&= (q-p)(t_{p,q}^2 + t_{p,q}) = (t_{p,q}+1)(q-p) t_{p,q} = (t_{p,q}+1)n_p \\
&= (t_{p,q}+1)(n_p+\alpha_{p,q}).
\end{align*}

\medskip

\noindent
\textbf{Case 2:} $\alpha_{p,q} > 0$.  

In this case, we have $n_p - \alpha_{p,q} = (q-p)t_{p,q}$, and
\[
\min_{m \in \mathbb{Z}} \left| m - \left(t_{p,q} + \frac{\alpha_{p,q}}{q-p} + \frac{1}{2} \right) \right|
= \left| t_{p,q} + 1 - \left(t_{p,q} + \frac{\alpha_{p,q}}{q-p} + \frac{1}{2} \right) \right|
= \left| \frac{1}{2} - \frac{\alpha_{p,q}}{q-p} \right|,
\]
since 
\[
0 < \frac{\alpha_{p,q}}{q-p} < 1. 
\]

Moreover, $t_{p,q}+1 > 0$ and 
\[
n_p-(t_{p,q}+1) = ((q-p)t_{p,q} + \alpha_{p,q})-(t_{p,q}+1) = (q-p-1)t_{p,q} + \alpha_{p,q}-1 \ge 0.
\]
If $t_{p,q}=0$, then $F_{p,q}(1,1) = F_{p,q}(1,t_{p,q}+1)$.  
If $t_{p,q} > 0$, we have
\begin{align*}
F_{p,q}(1,1) - F_{p,q}(1,t_{p,q}+1) 
&= (q-p)(t_{p,q}+1)^2 - (q-1)(t_{p,q}+1) + p-1 \\
&= t_{p,q} ((q-p)t_{p,q} + q - 2p + 1) \\
&= t_{p,q} ((n_p - \alpha_{p,q}) + (q-p) - (p - 1)) \\
&= t_{p,q} (-n_p - \alpha_{p,q} + (q-p)) \\
&\le -2 t_{p,q} \alpha_{p,q} < 0, 
\end{align*}
where we used Lemma \ref{lem(F(1,m))} and $0 < q-p \le (q-p)t_{p,q} = n_p - \alpha_{p,q}$ from Lemma \ref{lem(inequality[q])}.

Hence, $(1,t_{p,q}+1) \in \mathfrak{L}_p \cap R_{p,q}$, and
\begin{align*}
\max_{1 \le m \le n_p} F_{p,q}(1,m) 
&= F_{p,q}(1,t_{p,q}+1)\\
&= -(q-p)(t_{p,q}+1)^2 + (q-1)(t_{p,q}+1) \\
&= (t_{p,q}+1)(-(q-p)(t_{p,q}+1) + q-1) \\
&= (t_{p,q}+1)(-(q-p)t_{p,q} + p-1) \\
&= (t_{p,q}+1)(-(n_p - \alpha_{p,q}) + 2 n_p) \\
&= (t_{p,q}+1)(n_p+\alpha_{p,q}).
\end{align*}

Therefore, the desired result holds.
\end{proof}

\begin{rmk}
\label{rmk(Kuroda)} 
If $\alpha_{p,q}=0$, then by Lemma \ref{lem(F(1,m))} we have
\begin{align*}
\lefteqn{F_{p,q}(1, t_{p,q}+1) - F_{p,q}(1, t_{p,q})}\\
&= -(q-p)(t_{p,q}+1)^2 + (q-1)(t_{p,q}+1) - (-(q-p)t_{p,q}^2 + (q-1)t_{p,q}) \\
&= -(q-p)(2t_{p,q}+1) + (q-1) = -2(q-p)t_{p,q} - (q-p) + (q-1) \\
&= -2 n_p + p - 1 = -(p-1) + p - 1 = 0.
\end{align*}
Therefore, in this case, we have
\[
\max_{1\le m\le n_p}F_{p,q}(1,m)=F_{p,q}(1,t_{p,q}+1)=(t_{p,q}+1)(n_p+\alpha_{p,q}). 
\]
\end{rmk}

We recall that the set $\mathfrak{M}_{p,q}$ is defined by 
\[
\mathfrak{M}_{p,q}= \{(1,t_{p,q}+1)\}
\sqcup\{(a,m)\in(2\mathbb{Z}+1) \times \mathbb{Z} \;\big|\; 3 \le a<m\le n_p, F_{p,q}(a,m) \ge F_{p,q}(1,1)\}. 
\]

Here we prove Theorem~\ref{thm(odd-case)}.  
\begin{proof}[Proof of Theorem~\ref{thm(odd-case)}]
Assume that $p$ is odd and $pq + pr - qr = 1$. 
By Lemmas~\ref{lem(L-to-L')} and \ref{lem(m-less-than-a)}, we have 
\[
d(\Sigma (p,q,r))=\max_{(a,m)\in\mathfrak{L}_p\cap R_{p,q}}F_{p,q}(a,m)=\max_{(a,m)\in\mathfrak{N}'_{p,q}}F_{p,q}(a,m), 
\]
where
\[
\mathfrak{N}'_{p,q} := \{(a,m) \in (2\mathbb{Z}+1) \times \mathbb{Z} \;\big|\; 1 \le a < m \le n_p, \, F_{p,q}(a,m) \ge F_{p,q}(1,1)\}.
\]
Then, by Proposition \ref{prop(maxF(1,m))} and Remark \ref{rmk(Kuroda)}, we obtain
\begin{align*}
d(\Sigma (p,q,r))&=\max_{(a,m)\in \mathfrak{M}_{p,q}}F_{p,q}(a,m)\\
&\ge 
\max_{1 \le m \le n_p} F_{p,q}(1,m) = F_{p,q}(1, t_{p,q} + 1) = (t_{p,q}+1)(n_p+\alpha_{p,q}).
\end{align*}
Hence, the desired result follows.
\end{proof}

The points in $\mathfrak{L}_p$ used in Theorem \ref{thm([KS20]main)}, as well as the candidates for those in $\mathfrak{M}_{p,q}$ used in Theorem \ref{thm(odd-case)}, are shown in Figure \ref{fig(L-and-M)}.

\begin{figure}[htbp]
    \centering
\begin{overpic}[scale=0.6
]
{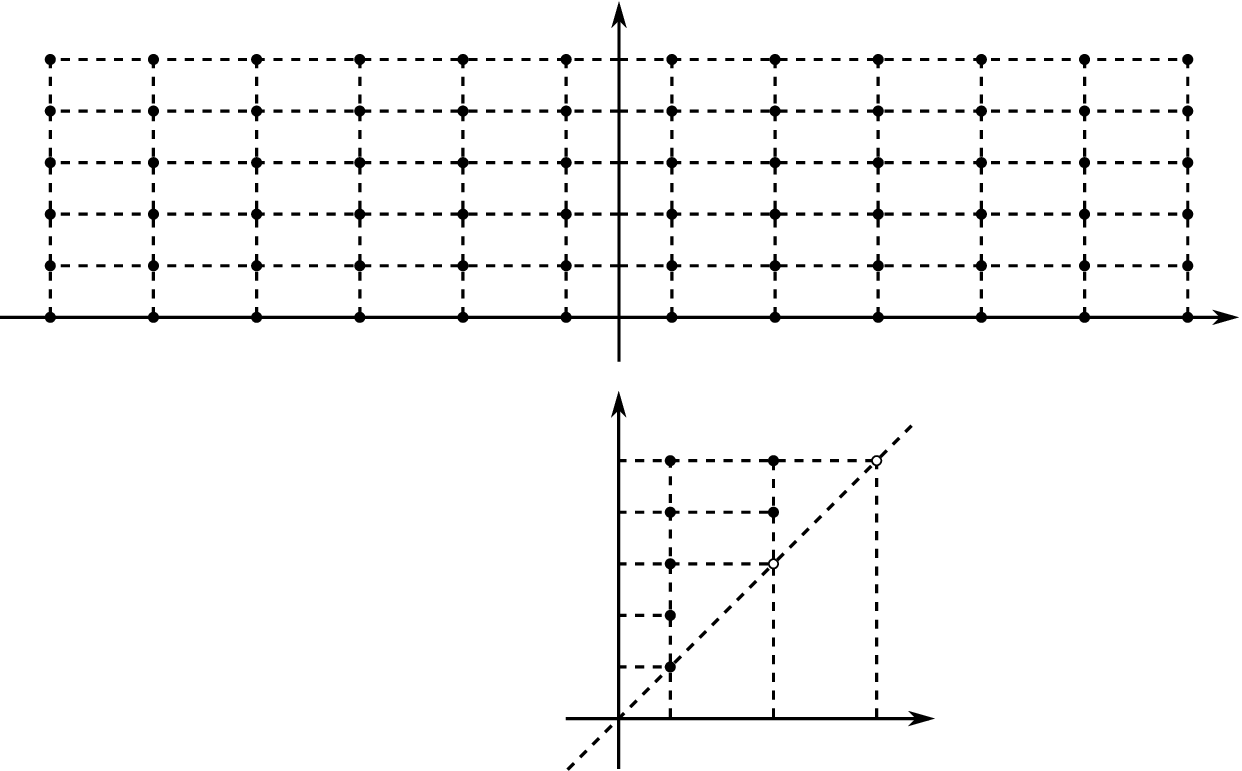}
    \put(46.5,33){$O$}
    \put(100.7,35.4){$a$}
    \put(45.5,60){$m$}

    \put(53.2,33){$1$}
    \put(61.2,33){$3$}
    \put(70.2,33){$5$}
    \put(77.2,33){$\cdots$}
    \put(84.2,33){$p-1$}
    \put(94.2,33){$p$}

    \put(42,33){$-1$}
    \put(34.2,33){$-3$}
    \put(26.2,33){$-5$}
    \put(19,33){$\cdots$}
    \put(4,33){$-(p-1)$}
    \put(-2,33){$-p$}

    \put(47.5,37.5){$1$}
    \put(47.5,42){$2$}
    \put(47.5,46){$3$}
    \put(47.5,49.1){\rotatebox{90}{$\cdots$}}
    \put(46,55){$n_p$}

    \put(50,1){$O$}
    \put(76.2,2.8){$a$}
    \put(45.5,29){$m$}
    \put(73.5,25){$m=a$}

    \put(53.2,1){$1$}
    \put(61.2,1){$3$}
    \put(64.6,1){$\cdots$}
    \put(70.2,1){$n_p$}

    \put(47.5,7){$1$}
    \put(47.5,11){$2$}
    \put(47.5,15.5){$3$}
    \put(47.5,18.5){\rotatebox{90}{$\cdots$}}
    \put(46,24.3){$n_p$}

\end{overpic}
\caption{Points in $\mathfrak{L}_{p}$ (top) and candidates in $\mathfrak{M}_{p,q}$ (bottom)}
\label{fig(L-and-M)}
\end{figure}

\begin{rmk} 
If $F_{p,q}(a,m) \le (t_{p,q}+1)(n_p+\alpha_{p,q})$ for all $(a,m) \in \mathfrak{M}_{p,q}$, then, by Theorem \ref{thm(odd-case)}, we have
\[
d(\Sigma(p,q,r)) = F_{p,q}(1,t_{p,q}+1) = (t_{p,q}+1)(n_p+\alpha_{p,q}). 
\]
\end{rmk}

As an example, we have the following corollary derived from Theorem \ref{thm(odd-case)}.

\begin{cor}
If $p$ is odd and $pq + pr - qr = 1$ with $3 \le p \le 23$ and $q-p \ge n_p$, then
\[
d(\Sigma(p,q,r)) = p-1.
\]
\end{cor}

\begin{proof}
We assume $p$ is odd satisfying $3\le p \le23$ and $1 \le a \le n_p$. 
By Lemma \ref{lem(inequality[q])}, we have 
\[
(p+a-2)-(q-p)=(a-1)+(2p-1)-q>0.
\]
By $a \le n_p =(p-1)/2$, we have 
\[
\frac{3p-1}{2}-a \ge p >0.
\]
If $q-p \ge n_p = (p-1)/2>0$, then we have 
\begin{eqnarray*}
\lefteqn{4(q-p)(F_{p,q}(1,1)-F_{p,q}(a,m))}\\
&=&(2(q-p)m-aq+1)^2-(q-a)^2+4(p-1)(q-p)\\
&\ge&-(q-a)^2+4(p-1)(q-p)\\
&=&-((p+a-2)-(q-p))^2+(p+a-2)^2-(p-a)^2\\
&\ge&-\left((p+a-2)-\frac{p-1}{2}\right)^2+(p+a-2)^2-(p-a)^2\\
&=&-a^2+(3p-1)a-\frac{p^2+10p-7}{4} \ge -\frac{p^2-26p+41}{4}. 
\end{eqnarray*}
Since $p^2-26p+41 \le 0$ and $3 \le p \le 23$ are equivalent. 
By Theorem \ref{thm(odd-case)}, we have the desired result above. 
\end{proof}

\subsection{Some properties}
\label{subsec(Some-properties)}
In this subsection we record several explicit formulas obtained from
Theorem~\ref{thm(odd-case)}. 
These formulas will be used repeatedly in the subsequent sections. 

Let $p,q$, and $r$ be integers satisfying $1<p<q<r$ and $pq+pr-qr=1$.
For such integers $p$ and $q$, the symbol $r_{p,q}$ denotes the corresponding value of $r$. 

Here we prove Theorem~\ref{thm(alpha0case)}.  
\begin{proof}[Proof of Theorem~\ref{thm(alpha0case)}]
Assume that $\alpha_{p,q}=0$. 
Then $n_p=l_{p,q}t_{p,q}$. 
Since $l_{p,q}=q-p$, a direct computation yields
\[
\frac{aq-1}{2(q-p)}
=
at_{p,q}
+\frac{a}{2}
+\frac{1}{2}
+
\frac{(a-1)t_{p,q}}{2n_p}. 
\]

Recall from Lemma~\ref{lem(F(a,m))} that
\[
F_{p,q}(a,m)
=
-\frac{(2l_{p,q}m-aq+1)^2-(q-a)^2}{4l_{p,q}}.
\]

Hence a candidate minimizing $F_{p,q}(a,m)$ is
\[
m
=
at_{p,q}
+\frac{a-1}{2}
+u,
\qquad 0\le u\le \dfrac{t_{p,q}}{2}+1.
\]
Since
\[
F_{p,q}\!\left(a,at_{p,q}+\frac{a-1}{2}\right)
=
F_{p,q}\!\left(a,at_{p,q}+\frac{a+1}{2}\right),
\]
it suffices to consider the case $u\ge1$.

Assume
\[
m
=
at_{p,q}
+\frac{a-1}{2}
+u,
\qquad 1\le u\le \dfrac{t_{p,q}}{2}+1.
\]
Then $2m=(2t_{p,q}+1)a-1+2u$. 

Using $l_{p,q}=n_p/t_{p,q}$, we compute
\begin{align*}
t_{p,q}^2(2l_{p,q}m-aq+1)^2
&=
((2u-1)n_p-t_{p,q}(a-1))^2 \\
&=
(2u-1)^2n_p^2
+t_{p,q}^2(a-1)^2
-2t_{p,q}n_p(2u-1)(a-1).
\end{align*}

Similarly,
\begin{align*}
t_{p,q}^2(q-a)^2
&=
((2t_{p,q}+1)n_p-t_{p,q}(a-1))^2 \\
&=
(2t_{p,q}+1)^2n_p^2
+t_{p,q}^2(a-1)^2 
-2t_{p,q}n_p(2t_{p,q}+1)(a-1).
\end{align*}

Moreover,
\[
t_{p,q}^2\cdot4l_{p,q}(t_{p,q}+1)n_p
=
4t_{p,q}(t_{p,q}+1)n_p^2.
\]

Taking the difference, we obtain
\begin{align*}
&t_{p,q}^2((2l_{p,q}m-aq+1)^2-(q-a)^2
+4l_{p,q}(t_{p,q}+1)n_p) \\
&=
((2u-1)^2-1)n_p^2
+4t_{p,q}n_p(a-1)(t_{p,q}-u+1).
\end{align*}

Since $1\le u\le t_{p,q}/2+1$, we have $t_{p,q}-u+1\ge0$ and $(2u-1)^2-1\ge0$.
Hence the above expression is nonnegative.

Therefore, $(t_{p,q}+1)n_p-F_{p,q}(a,m)\ge0$, and consequently $d(\Sigma(p,q,r))=(t_{p,q}+1)n_p$.
\end{proof}

Recall that
\[
D(p,q,r) =
\begin{cases}
(t_{p,q}+1)(n_p+\alpha_{p,q}), & \text{if $p$ is odd}, \\[2mm]
d(\Sigma(p,q,r)), & \text{if $p$ is even}.
\end{cases}
\]

By Theorem~\ref{thm(odd-case)}, if $p$ is odd, then
\[
d(\Sigma(p,q,r)) \ge D(p,q,r).
\]
All explicitly computed examples considered so far, including those in \cite{KS20}, satisfy
\[
d(\Sigma(p,q,r)) = D(p,q,r).
\]

We further define
\[
\chi_{p,q}(l_{p,q},\alpha_{p,q})
:=
\frac{l_{p,q}^2}{4}
-(\alpha_{p,q}+1)l_{p,q}
+\alpha_{p,q}^2+2.
\]

\begin{prop}
\label{prop(d=D-condition)}
Suppose that $p$ is odd and $pq + pr - qr = 1$. 
If $p \ge \chi_{p,q}(l_{p,q},\alpha_{p,q})$, then $d(\Sigma(p,q,r)) = D(p,q,r)$. 
\end{prop}

\begin{proof}
We first compute
\begin{align*}
4l_{p,q}D(p,q,r)
&=4l_{p,q}(t_{p,q}+1)(n_p+\alpha_{p,q}) \\
&=(2l_{p,q}t_{p,q}+2l_{p,q})(2n_p+2\alpha_{p,q}) \\
&=(p-1-2\alpha_{p,q}+2l_{p,q})(p-1+2\alpha_{p,q}),
\end{align*}
where we used $2l_{p,q}t_{p,q}=p-1-2\alpha_{p,q}$ and $2n_p=p-1$.

Hence
\begin{align*}
\lefteqn{4l_{p,q}D(p,q,r)-(p+l_{p,q}-a)^2} \\
&=(p-1-2\alpha_{p,q}+2l_{p,q})(p-1+2\alpha_{p,q})
 -(p+l_{p,q}-a)^2 \\
&=2(a-1)p-(l_{p,q}-a)^2
 +(2l_{p,q}-2\alpha_{p,q}-1)(2\alpha_{p,q}-1).
\end{align*}

Since $a\ge 3$, we obtain the estimate
\[
4l_{p,q}D(p,q,r)-(p+l_{p,q}-a)^2
\ge
4l_{p,q}D(p,q,r)-(p+l_{p,q}-3)^2.
\]

A direct computation gives
\[4l_{p,q}D(p,q,r)-(p+l_{p,q}-3)^2 
=4p-l_{p,q}^2+4(\alpha_{p,q}+1)l_{p,q}-4\alpha_{p,q}^2-8. 
\]

Therefore
\[
4l_{p,q}D(p,q,r)-(p+l_{p,q}-3)^2 \ge 0
\]
is equivalent to
\[
p \ge
\frac{l_{p,q}^2}{4}
-(\alpha_{p,q}+1)l_{p,q}
+\alpha_{p,q}^2+2
=
\chi_{p,q}(l_{p,q},\alpha_{p,q}).
\]

Under this condition we obtain
\[
4l_{p,q}D(p,q,r)-(p+l_{p,q}-a)^2 \ge 0.
\]

By the expression of $F_{p,q}(a,m)$ in Lemma~\ref{lem(F(a,m))}, 
this inequality implies
\[
F_{p,q}(a,m)\le D(p,q,r).
\]
Hence $d(\Sigma(p,q,r))=D(p,q,r)$. 
\end{proof}

Recall that $s_{p,q}=l_{p,q}/n_p$. 
By Lemma~\ref{lem(inequality[q])}, we have $0 < s_{p,q} < 2$. 
If $s_{p,q}=1$, then $d(\Sigma(p,q,r))=D(p,q,r)$ by Theorem~\ref{thm(Example[KS20])}~(3).

Suppose that $s_{p,q}<1$. 
Then $l_{p,q}<n_p$, and hence $p>2l_{p,q}+1$.
In this case, if 
\[
2l_{p,q}+3 \le p < \chi_{p,q}(l_{p,q},\alpha_{p,q}),
\]
then it may occur that 
\[
d(\Sigma(p,q,r)) \ne D(p,q,r)
\]
by comparison with Proposition~\ref{prop(d=D-condition)}. 

We now consider the case $s_{p,q}>1$. 

\begin{prop}
\label{prop(d=D-condition2)}
Suppose that $p$ is odd, $pq + pr - qr = 1$, and $s_{p,q}>1$.
If
\[
p \le 
\min\{
l_{p,q}+3+2\sqrt{2l_{p,q}},
\; 2l_{p,q}+1
\}, 
\]
then $d(\Sigma(p,q,r))=D(p,q,r)$. 
\end{prop}

\begin{proof}
Since $s_{p,q}>1$, we have $n_p<l_{p,q}$. 
Using $2n_p=p-1$, this inequality implies
\[
p<2l_{p,q}+1.
\]
Moreover, in this case we have $(t_{p,q}, \alpha_{p,q})=(0,n_p)$.

A direct computation gives
\[
4l_{p,q}D(p,q,r)-(p+l_{p,q}-3)^2
=
-p^2+2(l_{p,q}+3)p-(l_{p,q}-1)^2-8.
\]

Thus, the inequality 
\[
p \ge \chi_{p,q}(l_{p,q},\alpha_{p,q})
\]
is equivalent to
\[
4l_{p,q}D(p,q,r)-(p+l_{p,q}-3)^2 \ge 0,
\]
which in turn is equivalent to
\[
l_{p,q}+3-2\sqrt{2l_{p,q}}
\le p \le
l_{p,q}+3+2\sqrt{2l_{p,q}}.
\]

Since $2 \le l_{p,q} \le p-1$ and $p<2l_{p,q}+1$, 
the assumption on $p$ ensures that 
\[
p \ge \chi_{p,q}(l_{p,q},\alpha_{p,q}).
\]
Therefore, by Proposition~\ref{prop(d=D-condition)}, we conclude that
\[
d(\Sigma(p,q,r))=D(p,q,r).
\]
\end{proof}

\begin{lem}
\label{lem(Chi-properties)}
Suppose that $p$ is odd and $pq + pr - qr = 1$. 
The following symmetry holds:
\[
\chi_{p,q}(l_{p,q},\alpha_{p,q})
=
\chi_{p,q}(l_{p,q},l_{p,q}-\alpha_{p,q}).
\]
Moreover, assume that $1 \le \alpha_{p,q} \le l_{p,q}-1$. 
Then
\[
\chi_{p,q}(l_{p,q},\alpha_{p,q})
\le
\chi_{p,q}(l_{p,q},1).
\]
\end{lem}

\begin{proof}
By definition,
\[
\chi_{p,q}(l_{p,q},\alpha_{p,q})
=
\left(\frac{l_{p,q}^2}{4}-l_{p,q}+2\right)
+
\alpha_{p,q}(\alpha_{p,q}-l_{p,q}).
\]

First we prove the symmetry.
We compute
\begin{align*}
\chi_{p,q}(l_{p,q},l_{p,q}-\alpha_{p,q})
&=
\left(\frac{l_{p,q}^2}{4}-l_{p,q}+2\right)
+
(l_{p,q}-\alpha_{p,q})
((l_{p,q}-\alpha_{p,q})-l_{p,q}) \\
&=
\left(\frac{l_{p,q}^2}{4}-l_{p,q}+2\right)
+
\alpha_{p,q}(\alpha_{p,q}-l_{p,q}) \\
&=
\chi_{p,q}(l_{p,q},\alpha_{p,q}).
\end{align*}

Next we prove the inequality.
Completing the square yields
\begin{align*}
\chi_{p,q}(l_{p,q},\alpha_{p,q})
&=
\left(\frac{l_{p,q}^2}{4}-l_{p,q}+2\right)
+
\left(\alpha_{p,q}-\frac{l_{p,q}}{2}\right)^2
-\frac{l_{p,q}^2}{4} \\
&=
-l_{p,q}+2
+
\left(\alpha_{p,q}-\frac{l_{p,q}}{2}\right)^2.
\end{align*}

If $1 \le \alpha_{p,q} \le l_{p,q}-1$, then
\[
\left|\alpha_{p,q}-\frac{l_{p,q}}{2}\right|
\le
\left|1-\frac{l_{p,q}}{2}\right|.
\]
Hence
\[
\left(\alpha_{p,q}-\frac{l_{p,q}}{2}\right)^2
\le
\left(1-\frac{l_{p,q}}{2}\right)^2,
\]
which implies $\chi_{p,q}(l_{p,q},\alpha_{p,q})\le\chi_{p,q}(l_{p,q},1)$. 
\end{proof}

Here we prove Theorem~\ref{thm(1-19)}
\begin{proof}[Proof of Theorem~\ref{thm(1-19)}]
If $\alpha_{p,q}=0$, then the conclusion follows from 
Theorem~\ref{thm(alpha0case)}. 
Hence assume $\alpha_{p,q}\neq 0$. 
In this case $l_{p,q}\ge2$. 

If $s_{p,q}=1$, then the equality $d(\Sigma(p,q,r))=D(p,q,r)$ follows from Theorem~\ref{thm(Example[KS20])}~(3). 

\medskip

\noindent
{\bf Case 1: $s_{p,q}<1$.}

By Proposition~\ref{prop(d=D-condition)}, the inequality 
$d(\Sigma(p,q,r)) \ne D(p,q,r)$ can occur only if
\[
2l_{p,q}+3 \le p < \chi_{p,q}(l_{p,q},\alpha_{p,q}).
\]

If $l_{p,q}\le 16$, then 
\[
\chi_{p,q}(l_{p,q},\alpha_{p,q})
\le \chi_{p,q}(l_{p,q},1)
\le 2l_{p,q}+3
\]
by Lemma~\ref{lem(Chi-properties)}, and hence the above inequality cannot occur. 

If $17\le l_{p,q}\le 19$, a direct verification shows that there are no integers $p,q$ satisfying
\[
2l_{p,q}+3 \le p < \chi_{p,q}(l_{p,q},1)
\]
and
\[
r_{p,q}
=
p+\frac{p^2-1}{l_{p,q}}
\in\mathbb{Z}.
\]

Therefore, if $s_{p,q}<1$ and $1 \le q-p \le 19$, 
then $d(\Sigma(p,q,r))=D(p,q,r)$. 

\medskip

\noindent
{\bf Case 2: $s_{p,q}>1$.}

By Theorem~\ref{thm(Example[KS20])}~(2), (3) and Proposition~\ref{prop(d=D-condition2)}, the inequality 
$d(\Sigma(p,q,r)) \ne D(p,q,r)$ can occur only if 
\[
l_{p,q}+3+2\sqrt{2l_{p,q}}
<
p
\le
2l_{p,q}-3. 
\]

If $l_{p,q}\le 18$, then
\[
2l_{p,q}-3
\le
l_{p,q}+3+2\sqrt{2l_{p,q}},
\]
so the above inequality has no solution.

If $l_{p,q} = 19$, a direct verification shows that there are no integers 
$p,q$ satisfying
\[
l_{p,q}+3+2\sqrt{2l_{p,q}}
<
p
\le
2l_{p,q}-3
\]
and
\[
r_{p,q}
=
p+\frac{p^2-1}{l_{p,q}}
\in\mathbb{Z}.
\]

Therefore, if $s_{p,q}>1$ and $1 \le q-p \le 19$, we obtain $d(\Sigma(p,q,r))=D(p,q,r)$. 
\end{proof}

We present a new case in which $d(\Sigma(p,q,r))$ can be computed by means of Theorem~\ref{thm(1-19)}. 

\begin{cor}
Suppose that $p$ is odd and $pq + pr - qr = 1$. 
If $(t_{p,q},\alpha_{p,q})=(1,1)$, then
\[
d(\Sigma(p,q,r))=p+1.
\]
\end{cor}

\begin{proof}
Suppose that $(t_{p,q},\alpha_{p,q})=(1,1)$. 
Then $n_p=l_{p,q}+1$. 
Since $2n_p=p-1$, we obtain $p=2n_p+1$ and $q=3n_p$. 

Using the relation
\[
r_{p,q}=\frac{pq-1}{q-p},
\]
we compute 
\[
r_{p,q}
=
\frac{6n_p^2+3n_p-1}{n_p-1}
=
6n_p+9+\frac{8}{n_p-1}.
\]

Thus $r_{p,q}\in\mathbb Z$ if and only if $n_p-1$ divides $8$. 
In particular, $1 \le n_p-1 \le 8$, and hence
\[
1 \le q-p \le 8.
\]

Therefore Theorem~\ref{thm(1-19)} applies, and we obtain
\[
d(\Sigma(p,q,r))
=
(t_{p,q}+1)(n_p+\alpha_{p,q})
=
2(n_p+1)
=
2n_p+2
=
p+1.
\]
This completes the proof.
\end{proof}

\begin{rmk}
The $d$-invariant of the Brieskorn homology sphere 
$\Sigma(p,q,r)$ corresponding to the pretzel knot 
$P(-p,q,r)$ in Corollary~\ref{cor(End95)} 
is computable in all cases. 
More precisely, we have the following:

\begin{itemize}

\item[(1)]
Applying Theorem~\ref{thm(Example[KS20])}~(1), we obtain
\[
d(\Sigma(2k+1,4k+1,4k+3)) = 2k.
\]

\item[(2)]
Applying Theorem~\ref{thm(1-19)} in the case $l_{p,q}=2$, we obtain
\[
d(\Sigma(2k+1,2k+3,2k^2+4k+1))
=
\begin{cases}
2m(m+1) & (k=2m),\\
2(2m+1)(m+1) & (k=2m+1).
\end{cases}
\]

\item[(3)]
Applying Theorem~\ref{thm(1-19)} in the case $l_{p,q}=4$, we obtain
\[
d(\Sigma(2k+1,2k+5,k^2+3k+1))
=
\begin{cases}
4m(m+1) & (k=4m),\\
2(2m+1)(m+1) & (k=4m+1),\\
4(m+1)^2 & (k=4m+2),\\
2(2m+3)(m+1) & (k=4m+3).
\end{cases}
\]

\item[(4)]
Applying Theorem~\ref{thm(Example[KS20])}~(3), we obtain
\[
d(\Sigma(4k+1,6k+1,12k+5)) = 4k.
\]

\item[(5)]
Applying Theorem~\ref{thm(Example[KS20])}~(2), we obtain
\[
d(\Sigma(4k+3,6k+5,12k+7)) = 4k+2.
\]

\end{itemize}

In each of the above cases, we have
\[
d(\Sigma(p,q,r)) = D(p,q,r).
\]
\end{rmk}

\section{A parity-independent inequality}
\label{sec(A-parity-independent-inequality)}

Recall that
\[
D(p,q,r)
=
\begin{cases}
(t_{p,q}+1)(n_p+\alpha_{p,q}) & (p \text{ is odd}), \\[4pt]
d(\Sigma(p,q,r)) & (p \text{ is even}).
\end{cases}
\]

The invariant $D(p,q,r)$ provides a unified quantity that controls the size of the $d$-invariant $d(\Sigma(p,q,r))$ under the relation $pq+pr-qr=1$. 
In this section we establish bounds for $D(p,q,r)$ that do not depend on the parity of $p$. 

Here we prove Proposition~\ref{prop(D-ineq)}. 
\begin{proof}[Proof of Proposition~\ref{prop(D-ineq)}]
We divide the proof into two cases according to the parity of $p$.

\medskip
\noindent
\textbf{Case 1: $p$ is odd.}

Put 
\[
n_p=\frac{p-1}{2}.
\]
By definition, for $i=1,2$ we can write
\[
n_p=(q_i-p)t_{p,q_i}+\alpha_{p,q_i},
\qquad 
0 \le \alpha_{p,q_i} < q_i-p.
\]
Since $q_1 \ge q_2$, we have $q_1-p \ge q_2-p$, and hence
\[
t_{p,q_1} \le t_{p,q_2}.
\]

\medskip
\noindent
First, suppose that $t_{p,q_1}=t_{p,q_2}$. 
Then the division algorithm implies 
$\alpha_{p,q_1}=\alpha_{p,q_2}$,
and therefore
\[
D(p_2,q_2,r_2)-D(p_1,q_1,r_1)=0.
\]

\medskip
\noindent
Next, suppose that $t_{p,q_1}<t_{p,q_2}$. 
Then $t_{p,q_1}+1 \le t_{p,q_2}$, and we compute
\begin{align*}
 D(p_2,q_2,r_2)-D(p_1,q_1,r_1) 
&= (t_{p,q_2}+1)(n_p+\alpha_{p,q_2})
  -(t_{p,q_1}+1)(n_p+\alpha_{p,q_1}) \\
&\ge (t_{p,q_1}+2)(n_p+\alpha_{p,q_2})
  -(t_{p,q_1}+1)(n_p+\alpha_{p,q_1}) \\
&= (t_{p,q_1}+1)(\alpha_{p,q_2}-\alpha_{p,q_1})
   +(n_p+\alpha_{p,q_2}).
\end{align*}
Using $n_p=(q_1-p)t_{p,q_1}+\alpha_{p,q_1}$, we obtain
\begin{align*}
\lefteqn{(t_{p,q_1}+1)(\alpha_{p,q_2}-\alpha_{p,q_1})
  +(q_1-p)t_{p,q_1}
  +\alpha_{p,q_1}
  +\alpha_{p,q_2}} \\
&= t_{p,q_1}(q_1-p+\alpha_{p,q_2}-\alpha_{p,q_1})
  +2\alpha_{p,q_2}.
\end{align*}
Since $0 \le \alpha_{p,q_i} < q_i-p$, the right-hand side is non–negative. 
Hence
\[
D(p_1,q_1,r_1)\le D(p_2,q_2,r_2).
\]

\medskip
\noindent
For the extremal values, observe that:

If $q=2p-1$, then $q-p=p-1$ and the quotient of the division of $n_p$ by $q-p$ is $0$. 
Hence
\[
D(p,2p-1,2p+1)
=
(0+1)\left(n_p+n_p\right)
=
p-1.
\]

If $q=p+1$, then $q-p=1$, and therefore $t_{p,q}=n_p$ and $\alpha_{p,q}=0$. 
Thus
\[
D(p,p+1,p^2+p-1)
=
(n_p+1)n_p
=
\frac{p^2-1}{4}.
\]

\medskip
\noindent
\textbf{Case 2: $p$ is even.}

In this case,
\[
D(p,q,r)=d(\Sigma(p,q,r)).
\]
By Proposition~\ref{prop(even-case)} and a direct computation,
\begin{align*}
\lefteqn{4(q_1-p)(q_2-p)
(D(p_2,q_2,r_2)-D(p_1,q_1,r_1))} \\
&=
(-q_1q_2+p(q_1+q_2)-1)(q_1-q_2).
\end{align*}
Since $q_1 \ge q_2$, and by Lemma~\ref{lem(inequality[q])}, 
the right-hand side is non–negative. 
Hence
\[
D(p_1,q_1,r_1)\le D(p_2,q_2,r_2).
\]

Moreover, 
\[
D(p,2p-1,2p+1)=p,
\qquad
D(p,p+1,p^2+p-1)=\frac{p^2+2p}{4}.
\]

Combining these estimates, we obtain the desired inequalities.
\end{proof}

\begin{rmk}
From Proposition~\ref{prop(D-ineq)}, we observe that the lower bound 
\[
(t_{p,q}+1)(n_p+\alpha_{p,q}),
\]
which appears in Theorem~\ref{thm(odd-case)}, and the invariant $d(\Sigma(p,q,r))$ in the case where $p$ is even share analogous monotonicity properties with respect to $q$.

In particular, if $p$ is odd and 
\[
d(\Sigma(p,q,r)) = D(p,q,r),
\]
then the $d$-invariants satisfy the same type of inequality relations as in the case where $p$ is even. 
\end{rmk}

\begin{prop}
\label{prop(D=p-1)}
Suppose that $p$ is odd and $pq + pr - qr = 1$. 
Then
\[
D(p,q,r)=p-1
\quad\text{if and only if}\quad
q-p \ge n_p.
\]
\end{prop}

\begin{proof}
Let $t_{p,q}$ and $\alpha_{p,q}$ denote the quotient and remainder, respectively, when $n_p$ is divided by $q-p$. 
Thus
\[
n_p=(q-p)t_{p,q}+\alpha_{p,q},
\qquad
0\le \alpha_{p,q}<q-p,
\]
and
\[
D(p,q,r)=(t_{p,q}+1)(n_p+\alpha_{p,q}).
\]

First suppose that $q-p \ge n_p$.

If $q-p>n_p$, then $(t_{p,q},\alpha_{p,q})=(0,n_p)$.
If $q-p=n_p$, then $(t_{p,q},\alpha_{p,q})=(1,0)$.
In either case we obtain
\[
D(p,q,r)=2n_p=p-1.
\]

Conversely, suppose that $1\le q-p<n_p$.
Then either $t_{p,q}\ge2$, or $t_{p,q}=1$ with $\alpha_{p,q}\ge1$.

If $t_{p,q}\ge2$, then
\[
D(p,q,r)
\ge 3(n_p+\alpha_{p,q})
\ge 3n_p
>
2n_p
=
p-1.
\]

If $t_{p,q}=1$ and $\alpha_{p,q}\ge1$, then
\[
D(p,q,r)
=
2(n_p+\alpha_{p,q})
\ge
2(n_p+1)
>
2n_p
=
p-1.
\]

Hence $D(p,q,r)>p-1$ whenever $q-p<n_p$.
This completes the proof.
\end{proof}

At the end of this section, we present several examples in which 
\[
d(\Sigma(p,q,r))=D(p,q,r)
\]
holds.

\medskip

\noindent
{\bf The case $p=4k+1$.}

Suppose that $p=4k+1$. 
If $k\le2$, then $p=4k+1\le9$, and hence $q-p\le p-1\le8$. 
Therefore, by Theorem~\ref{thm(1-19)}, we obtain $d(\Sigma(p,q,r))=D(p,q,r)$. 

Assume henceforth that $k\ge3$.

\begin{prop}
Let $p=4k+1$ with $k\ge3$. 
Suppose that $pq + pr - qr = 1$. 
If $q-p\in\{1,2,4,8,k,2k,2k+1,4k\}$, then $d(\Sigma(p,q,r))=D(p,q,r)$. 
In particular, if both $k$ and $2k+1$ are prime, then these exhaust all possible values of $q-p$, and hence the equality always holds.
\end{prop}

\begin{proof}
Since $p=4k+1$, we have
\[
n_p=\frac{p-1}{2}=2k.
\]

If $q-p\le8$, then the equality follows from 
Theorem~\ref{thm(1-19)}.

If $q-p=k$ or $q-p=2k$, then from
\[
n_p=(q-p)t_{p,q}+\alpha_{p,q}
\]
we deduce that $\alpha_{p,q}=0$. 
Hence the conclusion follows from 
Theorem~\ref{thm(alpha0case)}.

If $q-p=2k+1=n_p+1$, then the equality follows from 
Theorem~\ref{thm(Example[KS20])}~(2).

If $q-p=4k=2n_p$, then the equality follows from 
Theorem~\ref{thm(Example[KS20])}~(1).

This completes the proof.
\end{proof}

\medskip

\noindent
{\bf The case $p=4k-1$.}

Suppose that $p=4k-1$. 
If $k\le2$, then $p=4k-1\le7$, and hence $q-p\le p-1\le6$. 
Therefore, by Theorem~\ref{thm(1-19)}, we obtain 
\[
d(\Sigma(p,q,r))=D(p,q,r).
\]

Assume henceforth that $k\ge3$.

\begin{prop}
Let $p=4k-1$ with $k\ge3$. 
Suppose that $pq + pr - qr = 1$. 
If $q-p\in\{1,2,4,8,k,2k-1,2k,4k-2\}$, then $d(\Sigma(p,q,r))=D(p,q,r)$. 
In particular, if both $k$ and $2k-1$ are prime, 
then these exhaust all possible values of $q-p$, 
and hence the equality always holds.
\end{prop}

\begin{proof}
Since $p=4k-1$, we have $n_p=2k-1$. 

If $q-p\le8$, then the equality follows from 
Theorem~\ref{thm(1-19)}.

If $q-p=k=(n_p+1)/2$, then the equality follows from 
Theorem~\ref{thm(Example[KS20])}~(4).

If $q-p=2k-1=n_p$, then the equality follows from 
Theorem~\ref{thm(Example[KS20])}~(3).

If $q-p=2k=n_p+1$, then the equality follows from 
Theorem~\ref{thm(Example[KS20])}~(2).

If $q-p=4k-2=2n_p$, then the equality follows from 
Theorem~\ref{thm(Example[KS20])}~(1).

This completes the proof.
\end{proof}

\section{Results for infinite kinds of infinite classes}
\label{sec(infinite-kinds-of-infinite-classes)}

In this section, we present explicit formulas for $d(\Sigma(p,q,r))$ for several infinite families.

Let $p, q,$ and $r$ be integers satisfying $1 < p < q < r$ and $pq + pr - qr = 1$. 
Recall that $n_p=(p-1)/2$, $l_{p,q}=q-p$, and $s_{p,q}=l_{p,q}/n_p$, and define $b_{a,m}:=m-a$. 

\subsection[The case $d(\Sigma(p,q,r))=D(p,q,r)=p-1$]{The case \texorpdfstring{$d(\Sigma(p,q,r))=D(p,q,r)=p-1$}{d(Sigma(p,q,r))=D(p,q,r)=p-1}}

Here we prove Theorem~\ref{thm(d=D=p-1-condition)}. 
\begin{proof}[Proof of Theorem~\ref{thm(d=D=p-1-condition)}]
A direct computation shows that 
\begin{align*}
\lefteqn{F_{p,q}(a,m)}\\
&=\displaystyle\frac{(2(a+1)-(2m-a-1)s_{p,q})(((2m-a+1)s_{p,q}-2(a-1))n_p-2(a-1))}{4s_{p,q}}, 
\end{align*}
and $F_{p,q}(1,1)=p-1$ by Lemma~\ref{lem(F(a,m))}. 

Assume that
\[
a+1 \le m < \frac{s_{p,q}+2}{2s_{p,q}}(a+1).
\]
Then
\[
(a+1)s_{p,q} \le (2m-a-1)s_{p,q} < 2(a+1),
\]
and one verifies that 
\[
F_{p,q}(a,m)\le F_{p,q}(1,1)
\]
whenever $s_{p,q}\ge 1/2$.

On the other hand, $F_{p,q}(a,m)>F_{p,q}(1,1)$ can occur only if
\[
(2(a+1)-(2m-a-1)s_{p,q})
((2m-a+1)s_{p,q}-2(a-1))
>8s_{p,q}.
\]
Since $s_{p,q}>0$, this inequality is equivalent to $F_{p,q}(a,m)>F_{p,q}(1,1)$. 

A straightforward algebraic manipulation shows that the above inequality is equivalent to
\[
\frac{2-s_{p,q}}{s_{p,q}}\frac{a-1}{2}
<
b_{a,m}
<
\frac{2-s_{p,q}}{s_{p,q}}\frac{a+1}{2}.
\]

Now assume that
\[
s_{p,q}=\frac{2u}{u+1}.
\]
Then
\[
\frac{2-s_{p,q}}{s_{p,q}}
=
\frac{1}{u}.
\]

Put $a_0 := (a-1)/2 \in \mathbb{N}$. 
Then the inequality becomes
\[
\frac{a_0}{u}
<
b_{a,m}
<
\frac{a_0+1}{u}.
\]

The length of this interval is
\[
\frac{a_0+1}{u}-\frac{a_0}{u}=\frac{1}{u}<1,
\]
and hence it contains no integer. 
Therefore there exists no pair $(a,m)$ satisfying
\[
F_{p,q}(a,m)>F_{p,q}(1,1).
\]
Consequently, $d(\Sigma(p,q,r))=F_{p,q}(1,1)=p-1$. 

This completes the proof.
\end{proof}

\begin{prop}
\label{prop(D=d-partI)}
Let $u$ and $v$ be positive integers. 
Suppose 
\[
(p_{u,v}, q_{u,v}, r_{u,v})
=
(4u(u+1)v-2u-1, 4u(2u+1)v-4u-1, 4(u+1)(2u+1)v-4u-3). 
\]
Then $d(\Sigma(p_{u,v},q_{u,v},r_{u,v}))=D(p_{u,v},q_{u,v},r_{u,v})=p_{u,v}-1=2(u+1)(2uv-1)$. 
\end{prop}

\begin{proof}
A direct computation shows that 
$p_{u,v}$, $q_{u,v}$, and $r_{u,v}$ are positive integers satisfying
$1<p_{u,v}<q_{u,v}<r_{u,v}$ and $p_{u,v}q_{u,v}+p_{u,v}r_{u,v}-q_{u,v}r_{u,v}=1$. 

Moreover, $p_{u,v}$ is odd. 
We compute $q_{u,v}-p_{u,v}=2u(2uv-1)$ and $p_{u,v}-1=2(u+1)(2uv-1)$.
Hence
\[
s_{p_{u,v},q_{u,v}}
=
\frac{2(q_{u,v}-p_{u,v})}{p_{u,v}-1}
=
\frac{2u}{u+1}.
\]

Therefore Theorem~\ref{thm(d=D=p-1-condition)} applies, and we obtain $d(\Sigma(p_{u,v},q_{u,v},r_{u,v}))=p_{u,v}-1$. 
This completes the proof.
\end{proof}

\begin{cor}
\label{cor(D=d-partI)}
Let $u$ and $v$ be positive integers. 
Suppose 
\[
(p_{u,v}, q_{u,v}, r_{u,v})
=
(4u(u+1)v-2u-1, 4u(2u+1)v-4u-1, 4(u+1)(2u+1)v-4u-3). 
\]
For each fixed positive integer $v$, then the group $\Theta_{\mathbb{Z}}^3$ contains the subgroup 
\[
\bigoplus_{u=1}^{\infty}
\mathbb Z[\Sigma(p_{u,v},q_{u,v},r_{u,v})]_{\sim_{\mathbb Z}}.
\]

Similarly, for each fixed positive integer $u$, the group $\Theta_{\mathbb{Z}}^3$ contains the subgroup 
\[
\bigoplus_{v=1}^{\infty}
\mathbb Z[\Sigma(p_{u,v},q_{u,v},r_{u,v})]_{\sim_{\mathbb Z}}.
\]
\end{cor}

\begin{proof}
A direct computation shows that 
$p_{u,v}q_{u,v}+p_{u,v}r_{u,v}-q_{u,v}r_{u,v}=1$. 
Moreover, for each fixed positive integer $v$, the sequence $(p_{u,v}q_{u,v}r_{u,v})_{u,v=1}^{\infty}$ is strictly increasing in $u$, and for each fixed positive integer $u$, it is strictly increasing in $v$.

Hence the assumptions of Theorem~\ref{thm(Fur90reinterpretation)} are satisfied in both cases.  It follows that the corresponding families are linearly independent in $\Theta_{\mathbb Z}^3$, which proves that they generate free abelian subgroups of infinite rank.
\end{proof}

\begin{rmk}
If $u>1$, then
\[
s_{p_{u,v},q_{u,v}}=\frac{2u}{u+1}\in(1,2).
\]
Moreover, if $(u_1,v_1)\neq (u_2,v_2)$, then $(p_{u_1,v_1}, q_{u_1,v_1})\neq(p_{u_2,v_2}, q_{u_2,v_2})$, so the pairs $(p_{u,v},q_{u,v})$ are mutually distinct.

Consequently, there exist infinitely many pairwise distinct Brieskorn homology spheres $\Sigma(p_{u,v},q_{u,v},r_{u,v})$ for which the $d$-invariant is explicitly computable. 
More precisely, by Proposition~\ref{prop(D=d-partI)} and Corollary~\ref{cor(D=d-partI)}, these manifolds satisfy 
\[
d(\Sigma(p_{u,v},q_{u,v},r_{u,v}))
=
D(p_{u,v},q_{u,v},r_{u,v})
\quad\text{and}\quad
s_{p_{u,v},q_{u,v}}\in(1,2).
\]
\end{rmk}

Here we pose the following question. 

\begin{que}
Does there exist a double sequence $((p_{n,k},q_{n,k},r_{n,k}))_{n,k=1}^{\infty}$ and a submodule $A$ of $\Theta_{\mathbb Z}^3$ such that
\[
\Theta_{\mathbb Z}^3
=
A
\oplus
\bigoplus_{n=1}^{\infty}
\bigoplus_{k=1}^{\infty}
\mathbb Z[\Sigma(p_{n,k},q_{n,k},r_{n,k})],
\]
and such that for each fixed positive integer $n$,
\[
d(\Sigma(p_{n,k_1},q_{n,k_1},r_{n,k_1}))
=
d(\Sigma(p_{n,k_2},q_{n,k_2},r_{n,k_2}))
\]
for all positive integers $k_1$ and $k_2$?
\end{que}

\subsection[The case $d(\Sigma(p,q,r))=D(p,q,r)\neq p-1$]{The case \texorpdfstring{$d(\Sigma(p,q,r))=D(p,q,r)\neq p-1$}{d(Sigma(p,q,r))=D(p,q,r) neq p-1}}

\begin{prop}
\label{prop(D=d-partII)}
Let $k_1$ and $k_2$ be positive integers. 
Suppose 
\[
(p_{k_1,k_2}, q_{k_1,k_2}, r_{k_1,k_2})
=
(2k_1k_2+1, (2k_1+1)k_2+1, 2k_1(2k_1+1)k_2+4k_1+1). 
\]
Then $d(\Sigma(p_{k_1,k_2},q_{k_1,k_2},r_{k_1,k_2}))=D(p_{k_1,k_2},q_{k_1,k_2},r_{k_1,k_2})=(k_1+1)k_1k_2$. 
\end{prop}

\begin{proof}
A direct computation shows that $p_{k_1,k_2}$, $q_{k_1,k_2}$, and $r_{k_1,k_2}$ are positive integers satisfying $1<p_{k_1,k_2}<q_{k_1,k_2}<r_{k_1,k_2}$ and $p_{k_1,k_2}q_{k_1,k_2}+p_{k_1,k_2}r_{k_1,k_2}-q_{k_1,k_2}r_{k_1,k_2}=1$. 

Moreover, $p_{k_1,k_2}$ is odd, $n_{p_{k_1,k_2}}=k_1k_2$, and $l_{p_{k_1,k_2},q_{k_1,k_2}}=q_{k_1,k_2}-p_{k_1,k_2}=k_2$. 
By definitions of $t_{p_{k_1,k_2},q_{k_1,k_2}}$ and $\alpha_{p_{k_1,k_2},q_{k_1,k_2}}$, we have $t_{p_{k_1,k_2},q_{k_1,k_2}} = k_1$ and $\alpha_{p_{k_1,k_2},q_{k_1,k_2}}=0$. 

Hence, by Theorem~\ref{thm(alpha0case)}, 
\[d(\Sigma(p_{k_1,k_2},q_{k_1,k_2},r_{k_1,k_2}))
=
(t_{p_{k_1,k_2},q_{k_1,k_2}}+1)n_{p_{k_1,k_2}}
=
(k_1+1)k_1k_2. 
\]
This completes the proof.
\end{proof}

\begin{rmk}
Proposition~\ref{prop(D=d-partII)} provides a complete list of all explicit examples satisfying the assumptions of Theorem~\ref{thm(alpha0case)}.
\end{rmk}

\begin{cor}
\label{cor(D=d-partII)}
Let $k_1$ and $k_2$ be positive integers. 
Suppose 
\[
(p_{k_1,k_2}, q_{k_1,k_2}, r_{k_1,k_2})
=
(2k_1k_2+1, (2k_1+1)k_2+1, 2k_1(2k_1+1)k_2+4k_1+1). 
\]
For each fixed positive integer $k_2$, the group $\Theta_{\mathbb{Z}}^3$ contains the subgroup 
\[
\bigoplus_{k_1=1}^{\infty}
\mathbb Z[\Sigma(p_{k_1,k_2},q_{k_1,k_2},r_{k_1,k_2})]_{\sim_{\mathbb Z}}. 
\]

Similarly, for each fixed positive integer $k_1$, the $\Theta_{\mathbb{Z}}^3$ contains the subgroup 
\[
\bigoplus_{k_2=1}^{\infty}
\mathbb Z[\Sigma(p_{k_1,k_2},q_{k_1,k_2},r_{k_1,k_2})]_{\sim_{\mathbb Z}}. 
\]
\end{cor}

\begin{proof}
The proof is identical to that of Corollary~\ref{cor(D=d-partI)}. 
\end{proof}

\begin{rmk}
If $k_1>1$, then
\[
s_{p_{k_1,k_2},q_{k_1,k_2}}=\frac{1}{k_1}\in(0,1).
\]
Moreover, if $(k_{1,1},k_{2,1})\neq(k_{1,2},k_{2,2})$, then $(p_{k_{1,1},k_{2,1}}, q_{k_{1,1},k_{2,1}})\neq(p_{k_{1,2},k_{2,2}}, q_{k_{1,2},k_{2,2}})$,
so the pairs $(p_{k_1,k_2},q_{k_1,k_2})$ are mutually distinct.

Consequently, there exist infinitely many pairwise distinct Brieskorn homology spheres $\Sigma(p_{k_1,k_2},q_{k_1,k_2},r_{k_1,k_2})$ for which the $d$-invariant is explicitly computable. 
More precisely, by Proposition~\ref{prop(D=d-partII)} and Corollary~\ref{cor(D=d-partII)}, these manifolds satisfy 
\[
d(\Sigma(p_{k_1,k_2},q_{k_1,k_2},r_{k_1,k_2}))
=
D(p_{k_1,k_2},q_{k_1,k_2},r_{k_1,k_2})
\quad\text{and}\quad
s_{p_{k_1,k_2},q_{k_1,k_2}}\in(0,1).
\]
\end{rmk}

We pose the following question. 

\begin{que}
Does there exist an infinite sequence $((p_n,q_n,r_n))_{n=1}^{\infty}$ such that \\
$(\Sigma(p_n,q_n,r_n))_{n=1}^{\infty}$ generates a $\mathbb{Z}^{\infty}$-summand of
$\Theta_{\mathbb Z}^3$ and 
\[
d(\Sigma(p_n,q_n,r_n))
=
D(p_n,q_n,r_n)
>
p_n-1 \quad \text{for all $n$?}
\]
\end{que}

\section[Infinitely many examples of $d(\Sigma(p,q,r)) \neq D(p,q,r)$]{Infinitely many examples of  \texorpdfstring{$d(\Sigma(p,q,r)) \neq D(p,q,r)$}{d(Sigma(p,q,r)) neq D(p,q,r)}}
\label{sec(d-neq-D)}

Let $p,q$, and $r$ be integers satisfying $1 < p < q < r$ and $pq+pr-qr=1$. 

Recall the notation $n_p=(p-1)/2$, $l_{p,q}=q-p$, $s_{p,q}=l_{p,q}/n_p$, and $b_{a,m}=m-a$. 
Write
\[
n_p=t_{p,q}l_{p,q}+\alpha_{p,q},
\qquad
0\le \alpha_{p,q}<l_{p,q},
\]
where $t_{p,q}$ and $\alpha_{p,q}$ denote the quotient and the remainder,
respectively.

Up to the previous section, we have considered only the case where $d(\Sigma(p,q,r))
=
D(p,q,r)
=
(t_{p,q}+1)(n_p+\alpha_{p,q})$.

In this section, we provide a sufficient condition for the inequality 
$d(\Sigma(p,q,r)) \neq D(p,q,r)$ to hold. 

Here we prove Theorem~\ref{thm(D-neq-d-condition)}
\begin{proof}[Proof of Theorem~\ref{thm(D-neq-d-condition)}]
First, observe that
\[
F_{p,q}(a,m)
=
\frac{(2(a+1)-(2m-a-1)s_{p,q})
(((2m-a+1)s_{p,q}-2(a-1))n_p-2(a-1))}{4s_{p,q}}.
\]

For sufficiently large $p$, there exists $(a,m)\in\mathfrak{M}_{p,q}$ such that
\[
2s_{p,q}m \ge (s_{p,q}+2)(a+1).
\]
Set $b_{a,m} := m-a$.

Since $s_{p,q} \neq 2u/(u+1)$ for any $u\in\mathbb{N}$, there exists $(a,m)\in\mathfrak{M}_{p,q}$ satisfying 
\[
\frac{2-s_{p,q}}{s_{p,q}}\cdot\frac{a-1}{2}
<
b_{a,m}
<
\frac{2-s_{p,q}}{s_{p,q}}\cdot\frac{a+1}{2}.
\]
A direct computation shows that this inequality is equivalent to
\[
\frac{(2(a+1)-(2m-a-1)s_{p,q})
((2m-a+1)s_{p,q}-2(a-1))}{4s_{p,q}}
> 2.
\]
Combining this with the expression of $F_{p,q}(a,m)$, we obtain
\[
F_{p,q}(a,m)
>2n_p-\frac{(2(a+1)-(2m-a-1)s_{p,q})
(a-1)}{2s_{p,q}}\ge 2n_p=p-1
\]
under the above choice of $(a,m)$. 
By the assumption $s_{p,q}\in(1,2)$, we have $(t_{p,q}, \alpha_{p,q})=(0,n_p)$, and hence
\[
D(p,q,r)=(0+1)(n_p+0) = p-1,
\]
This completes the proof.
\end{proof}

\begin{prop}
\label{prop(D-neq-dPartI)}
Let $t$ and $k$ be positive integers satisfying $(t,k)\neq(2,1)$.
Define 
\[
(p_{t,k,1}, q_{t,k,1}, r_{t,k,1}) = (4t(2t+1)k+4t+1, (2t+1)(6t+1)k+6t+2, 4t(6t+1)k+12t+1). 
\]
Then $d(\Sigma(p_{t,k,1},q_{t,k,1},r_{t,k,1})) > p_{t,k,1}-1$. 
\end{prop}

\begin{proof}
First observe that $r_{t,k,1}=p_{t,k,1}+8t(2tk+1)$,
and $p_{t,k,1}q_{t,k,1}+p_{t,k,1}r_{t,k,1}-q_{t,k,1}r_{t,k,1}=1$. 

Since 
$
l_{p_{t,k,1},q_{t,k,1}}
=q_{t,k,1}-p_{t,k,1}
=(2t+1)^2k+2t+1
>
2t(2t+1)k+2t
=n_{p_{t,k,1}}$, it suffices, by Proposition~\ref{prop(D=p-1)}, to prove that
\[
F_{p_{t,k,1},q_{t,k,1}}(3,4)
>
F_{p_{t,k,1},q_{t,k,1}}(1,1)
=
D(p_{t,k,1},q_{t,k,1},r_{t,k,1})
=
p_{t,k,1}-1.
\]

A direct computation gives
\begin{align*}
F_{p_{t,k,1},q_{t,k,1}}(3,4)
&=\frac{2}{s_{p_{t,k,1},q_{t,k,1}}}(2-s_{p_{t,k,1},q_{t,k,1}})((3s_{p_{t,k,1},q_{t,k,1}}-2)n_{t,k,1}-2)\\
&=\frac{2}{4t^2s_{p_{t,k,1},q_{t,k,1}}}(4t-2ts_{p_{t,k,1},q_{t,k,1}})((6ts_{p_{t,k,1},q_{t,k,1}}-4t)n_{p_{t,k,1}}-4t)\\
&=\frac{2}{2t(2t+1)}(4t-(2t+1))((6t+3-4t)n_{p_{t,k,1}}-4t)\\
&=\frac{1}{t(2t+1)}((2t-1)(2t+3)n_{p_{t,k,1}}-4t(2t-1))\\
&=2n_{p_{t,k,1}}+\frac{1}{t(2t+1)}((2t-3)n_{p_{t,k,1}}-4t(2t-1))\\
&=p_{t,k,1}-1+2((2t-3)k-1)
\end{align*}
by $s_{p_{t,k,1},q_{t,k,1}}=(2t+1)/2t$. 

If $(t,k)\neq(2,1)$, then $(2t-3)k-1>0$, and hence $F_{p_{t,k,1},q_{t,k,1}}(3,4) > p_{t,k,1}-1$. 

This completes the proof.
\end{proof}

\begin{cor}
\label{cor(D-neq-dPartI)}
Let $t$ and $k$ be positive integers satisfying $(t,k)\neq(2,1)$.
Define 
\[
(p_{t,k,1}, q_{t,k,1}, r_{t,k,1}) = (4t(2t+1)k+4t+1, (2t+1)(6t+1)k+6t+2, 4t(6t+1)k+12t+1). 
\]
For each fixed positive integer $k\neq1$, the group $\Theta_{\mathbb{Z}}^3$ contains the subgroup 
\[
\bigoplus_{t=1}^{\infty}
\mathbb Z[\Sigma(p_{t,k,1},q_{t,k,1},r_{t,k,1})]_{\sim_{\mathbb Z}}. 
\]

Similarly, for each fixed positive integer $t\neq2$, the group $\Theta_{\mathbb{Z}}^3$ contains the subgroup 
\[
\bigoplus_{k=1}^{\infty}
\mathbb Z[\Sigma(p_{t,k,1},q_{t,k,1},r_{t,k,1})]_{\sim_{\mathbb Z}}. 
\]
\end{cor}

\begin{proof}
The proof is identical to that of Corollary~\ref{cor(D=d-partI)}. 
\end{proof}

\begin{rmk}
We have
\[
s_{p_{t,k,1},q_{t,k,1}}
=
1+\frac{1}{2t}\in(1,2).
\]
Moreover, if $(t_1,k_1)\neq (t_2,k_2)$, then $(p_{t_1,k_1,1}, q_{t_1,k_1,1})
\neq
(p_{t_2,k_2,1}, q_{t_2,k_2,1})$, and hence the pairs $(p_{t,k,1},q_{t,k,1})$ are mutually distinct.

Consequently, by Proposition~\ref{prop(D-neq-dPartI)} and
Corollary~\ref{cor(D-neq-dPartI)}, there exist infinitely many pairwise
distinct Brieskorn homology $3$-spheres $\Sigma(p_{t,k,1},q_{t,k,1},r_{t,k,1})$ satisfying 
\[
d(\Sigma(p_{t,k,1},q_{t,k,1},r_{t,k,1}))
>
D(p_{t,k,1},q_{t,k,1},r_{t,k,1})
\quad \text{and} \quad
s_{p_{t,k,1},q_{t,k,1}}\in(1,2).
\]
\end{rmk}

\begin{rmk}
If $(t,k)=(2,1)$, then $(p_{2,1,1},q_{2,1,1},r_{2,1,1})=(49,79,129)$.
In this case,
\[
d(\Sigma(49,79,129))
=
D(49,79,129)
=
48.
\]
\end{rmk}

\begin{prop}
\label{prop(D-neq-dPartII)}
Let $t$ and $k$ be positive integers satisfying 
$t\neq1$, $tk \in 2\mathbb{Z}$, and $(t,k)\in \mathbb{N}^2 \setminus \{(3,2),(4,1)\}$. 
Define 
\[
(p_{t,k,2}, q_{t,k,2}, r_{t,k,2}) = (t(2t-1)k+4t-1, t(3t-1)k+6t-1, (2t-1)(3t-1)k+12t-5). 
\]
Then $d(\Sigma(p_{t,k,2},q_{t,k,2},r_{t,k,2}))>p_{t,k,2}-1$. 
\end{prop}

\begin{proof}
First observe that $r_{t,k,2}=p_{t,k,2}+(2t-1)((2t-1)k+4)$, and that $p_{t,k,2}q_{t,k,2}+p_{t,k,2}r_{t,k,2}-q_{t,k,2}r_{t,k,2}=1$. 

Since $l_{p_{t,k,2},q_{t,k,2}}=q_{t,k,2}-p_{t,k,2}=t(tk+2)>(2t-1)(tk+2)/2=n_{p_{t,k,2}}$, it suffices, by Proposition~\ref{prop(D=p-1)}, to prove that $F_{p_{t,k,2},q_{t,k,2}}(3,4)>F_{p_{t,k,2},q_{t,k,2}}(1,1)=D(p_{t,k,2},q_{t,k,2},r_{t,k,2})=p_{t,k,2}-1$. 

A direct computation gives
\begin{align*}
F_{p_{t,k,2},q_{t,k,2}}(3,4)
&=\frac{2}{s_{p_{t,k,2},q_{t,k,2}}}(2-s_{p_{t,k,2},q_{t,k,2}})((3s_{p_{t,k,2},q_{t,k,2}}-2)n-2)\\
&=\frac{2}{(2t-1)^2s_{p_{t,k,2},q_{t,k,2}}}(4t-2-(2t-1)s_{p_{t,k,2},q_{t,k,2}})\\
&\cdot((3(2t-1)s_{p_{t,k,2},q_{t,k,2}}-4t+2)n_{p_{t,k,2}}-4t+2)\\
&=\frac{2}{2t(2t-1)}(4t-2-2t)((6t-4t+2)n_{p_{t,k,2}}-4t+2)\\
&=\frac{1}{t(2t-1)}(2t-2)(2(t+1)n_{p_{t,k,2}}-4t+2)\\
&=\frac{1}{t(2t-1)}(4(t^2-1)n_{p_{t,k,2}}-4(t-1)(2t-1))\\
&=\frac{2}{t(2t-1)}((t(2t-1)+(t-2))n_{p_{t,k,2}}-2(t-1)(2t-1))\\
&=2n_{p_{t,k,2}}+\frac{2}{t(2t-1)}((t-2)n_{p_{t,k,2}}-2(t-1)(2t-1))\\
&=p_{t,k,2}-1+(t-2)k-2
\end{align*}
by $s_{p_{t,k,2},q_{t,k,2}}=2t/(2t-1)$. 

If $t\neq1$, $tk\in2\mathbb{Z}$, and $(t,k)\in\mathbb{N}^2-\{(3,2), (4,1)\}$, then $(t-2)k-2>0$, and hence
\[
F_{p_{t,k,2},q_{t,k,2}}(3,4)
>
p_{t,k,2}-1.
\]
This completes the proof. 
\end{proof}

\begin{cor}
\label{cor(D-neq-dPartII)}
Let $t$ and $k$ be positive integers satisfying 
$t\neq1$, $tk \in 2\mathbb{Z}$, and $(t,k)\in \mathbb{N}^2 \setminus \{(3,2),(4,1)\}$. 
Define 
\[
(p_{t,k,2}, q_{t,k,2}, r_{t,k,2}) = (t(2t-1)k+4t-1, t(3t-1)k+6t-1, (2t-1)(3t-1)k+12t-5). 
\]
For each fixed positive integer $k\in2\mathbb{N}-\{2\}$, the group $\Theta_{\mathbb{Z}}^3$ contains the subgroup 
\[
\bigoplus_{t=2}^{\infty}
\mathbb Z[\Sigma(p_{t,k,2},q_{t,k,2},r_{t,k,2})]_{\sim_{\mathbb Z}}.
\]

Similarly, for each fixed positive integer $t\in2\mathbb{N}-\{4\}$, the group $\Theta_{\mathbb{Z}}^3$ contains the subgroup 
\[
\bigoplus_{k=1}^{\infty}
\mathbb Z[\Sigma(p_{t,k,2},q_{t,k,2},r_{t,k,2})]_{\sim_{\mathbb Z}}. 
\]
\end{cor}

\begin{proof}
The proof is identical to that of Corollary~\ref{cor(D=d-partI)}. 
\end{proof}

\begin{rmk}
If $t>1$, then 
\[
s_{p_{t,k,2},q_{t,k,2}}
=
1+\frac{1}{2t-1}\in(1,2).
\]
Moreover, if $(t_1,k_1)\neq (t_2,k_2)$, then $(p_{t_1,k_1,2}, q_{t_1,k_1,2})\neq(p_{t_2,k_2,2}, q_{t_2,k_2,2})$, and hence the pairs $(p_{t,k,2},q_{t,k,2})$ are mutually distinct. 

Consequently, by Proposition~\ref{prop(D-neq-dPartII)} and
Corollary~\ref{cor(D-neq-dPartII)}, there exist infinitely many pairwise
distinct Brieskorn homology $3$-spheres $\Sigma(p_{t,k,2},q_{t,k,2},r_{t,k,2})$ satisfying
\[
d(\Sigma(p_{t,k,2},q_{t,k,2},r_{t,k,2}))
>
D(p_{t,k,2},q_{t,k,2},r_{t,k,2})
\quad \text{and} \quad
s_{p_{t,k,2},q_{t,k,2}}\in(1,2).
\]
\end{rmk}

\begin{rmk}
The excluded parameter values correspond to the following exceptional cases.

\medskip
\noindent
(1) Suppose that $t=1$ and $k\in 2\mathbb{N}$. 
Then $(p_{1,k,2},q_{1,k,2},r_{1,k,2})=(k+3,2k+5,2k+7)$. 
In this case we have $s_{p_{1,k,2},q_{1,k,2}}=2$, and hence
\[
d(\Sigma(p_{1,k,2},q_{1,k,2},r_{1,k,2}))
=
D(p_{1,k,2},q_{1,k,2},r_{1,k,2})
=
k+2
\]
by Theorem~\ref{thm(Example[KS20])}~(1).

\medskip
\noindent
(2) When $(t,k)=(3,2)$, we obtain $(p_{3,2,2},q_{3,2,2},r_{3,2,2})=(41,65,111)$, and therefore 
\[
d(\Sigma(41,65,111))=D(41,65,111)=40. 
\]

\medskip
\noindent
(3) When $(t,k)=(4,1)$, we obtain $(p_{4,1,2},q_{4,1,2},r_{4,1,2})=(43,67,120)$.
Hence
\[
d(\Sigma(43,67,120))
=
D(43,67,120)
=
42.
\]
\end{rmk}

We pose the following question. 

\begin{que}
Does there exist an infinite sequence $((p_n,q_n,r_n))_{n=1}^{\infty}$ such that \\
$(\Sigma(p_n,q_n,r_n))_{n=1}^{\infty}$ generates a $\mathbb{Z}^{\infty}$-summand of
$\Theta_{\mathbb Z}^3$ and 
\[
d(\Sigma(p_n,q_n,r_n))
>
D(p_n,q_n,r_n) \quad \text{for all $n$?}
\]
\end{que}

\section[The Fibonacci cases $d(\Sigma(F_{2k+1},F_{2k+2},F_{2k+3}))$]{The Fibonacci cases \texorpdfstring{$d(\Sigma(F_{2k+1},F_{2k+2},F_{2k+3}))$}{d(Sigma(F\_{2k+1},F\_{2k+2},F\_{2k+3}))}}
\label{sec(Fibonacci-cases)}

Let $F_m$ denote the $m$-th Fibonacci number. 
Recall that $F_m$ is even if and only if $m\in 3\mathbb Z$. 
In this section we study the values of $d(\Sigma(F_{2k+1},F_{2k+2},F_{2k+3}))$ and compare them with $D(F_{2k+1},F_{2k+2},F_{2k+3})$. 

If $2k+1\in 3\mathbb Z$, then Proposition~\ref{prop(even-case)} gives 
\[
d(\Sigma(F_{2k+1},F_{2k+2},F_{2k+3}))
=D(\Sigma(F_{2k+1},F_{2k+2},F_{2k+3}))
=\frac{F_{2k+2}+F_{2k+3}}{4}
=\frac{F_{2k+4}}{4}. 
\]

Assume now that $2k+1\notin 3\mathbb Z$.
Then $l_{F_{2k+1},F_{2k+2}}
=q_{F_{2k+1},F_{2k+2}}-p_{F_{2k+1},F_{2k+2}}
=F_{2k+2}-F_{2k+1}
=F_{2k}$. 
Moreover,
\[
F_{2k}-\frac{F_{2k+1}-1}{2}
=\frac{2F_{2k}-F_{2k+1}+1}{2}
=\frac{F_{2k}-F_{2k-1}+1}{2}>0,
\]
where we used $F_{2k+1}=F_{2k}+F_{2k-1}$.
Hence the quotient obtained by dividing
\[
n_{F_{2k+1}}=\frac{F_{2k+1}-1}{2}
\]
by $l_{F_{2k+1},F_{2k+2}}=F_{2k}$ is $0$, and the remainder equals
$(F_{2k+1}-1)/2$.
Therefore,
\[
D(F_{2k+1},F_{2k+2},F_{2k+3})
=(0+1)\left(
\frac{F_{2k+1}-1}{2}
+\frac{F_{2k+1}-1}{2}
\right)
=F_{2k+1}-1.
\]
Note that $F_{F_{2k+1},F_{2k+2}}(1,1)=F_{2k+1}-1$. 

Here we prove Proposition~\ref{prop(Fibonacci-case)}. 
\begin{proof}[Proof of Proposition~\ref{prop(Fibonacci-case)}]
Since $F_{2k+1}$ is odd, we have $2k+1\in\{6j\pm1\}$ for some $j\in\mathbb Z$.
From
\[
l_{F_{2k+1},F_{2k+2}}=F_{2k+2}-F_{2k+1}=F_{2k}
\quad\text{and}\quad
s_{F_{2k+1},F_{2k+2}}
=
\frac{l_{F_{2k+1},F_{2k+2}}}{n_{F_{2k+1}}}
=
\frac{2F_{2k}}{F_{2k+1}-1},
\]
we compute
\begin{align*}
&F_{F_{2k+1},F_{2k+2}}(3,4)\\
&=\frac{2}{s_{F_{2k+1},F_{2k+2}}}(2-s_{F_{2k+1},F_{2k+2}})((3s_{F_{2k+1},F_{2k+2}}-2)n_{F_{2k+1}}-2)\\
&=\frac{F_{2k+1}-1}{F_{2k}}\left(2-\frac{2F_{2k}}{F_{2k+1}-1}\right)\left(\left(3\frac{2F_{2k}}{F_{2k+1}-1}-2\right)k-2\right)\\
&=\frac{2((F_{2k+1}-1)-F_{2k})((6F_{2k}-2(F_{2k+1}-1))n-2(F_{2k+1}-1))}{F_{2k}(F_{2k+1}-1)}\\
&=\frac{2(F_{2k-1}-1)((3F_{2k}-F_{2k+1}+1) \cdot 2n-2(F_{2k+1}-1))}{F_{2k}(F_{2k+1}-1)}\\
&=\frac{2(F_{2k-1}-1)((2F_{2k}-F_{2k-1}+1)(F_{2k+1}-1)-2(F_{2k+1}-1))}{F_{2k}(F_{2k+1}-1)}\\
&=\frac{2(F_{2k-1}-1)(2F_{2k}-F_{2k-1}-1)}{F_{2k}}\\
&=\frac{2(F_{2k-1}-1)(F_{2k}+F_{2k-2}-1)}{F_{2k}}.
\end{align*}

Hence 
\begin{align*}
&F_{2k}(F_{F_{2k+1},F_{2k+2}}(3,4)-F_{F_{2k+1},F_{2k+2}}(1,1))\\
&=2(F_{2k-1}-1)(F_{2k}+F_{2k-2}-1)-F_{2k}(F_{2k+1}-1)\\
&=(2F_{2k-1}-F_{2k+1}-1)F_{2k}+2(F_{2k-1}-1)(F_{2k-2}-1)\\
&=(F_{2k-1}-F_{2k}-1)F_{2k}+2(F_{2k-1}-1)(F_{2k-2}-1)\\
&=2(F_{2k-1}-1)(F_{2k-2}-1)-(F_{2k-2}+1)F_{2k}\\
&=(2F_{2k-1}-2)(F_{2k-2}-1)-(F_{2k-2}-1)F_{2k}-2F_{2k}\\
&=(2F_{2k-1}-F_{2k}-2)(F_{2k-2}-1)-2F_{2k}\\
&=(F_{2k-1}-F_{2k-2}-2)(F_{2k-2}-1)-2F_{2k}\\
&=(F_{2k-3}-2)(F_{2k-2}-1)-2F_{2k}\\
&=F_{2k-3}F_{2k-2}-F_{2k-3}-2F_{2k-2}+2-2F_{2k}\\
&=F_{2k-3}F_{2k-2}-F_{2k-1}-F_{2k-2}+2-2F_{2k}\\
&=F_{2k-3}F_{2k-2}-3F_{2k}+2.
\end{align*}

If $k=5$, then
\[
F_{10}(F_{F_{11},F_{12}}(3,4)-F_{F_{11},F_{12}}(1,1))
=
F_7F_8-3F_{10}+2
=110>0 .
\]

Assume $k\ge6$.
Since $F_{j}=F_{j-1}+F_{j-2}<2F_{j-1}$ for $j>3$, we obtain 
$F_{2k}
<
4F_{2k-3}$.
Hence
\begin{align*}
F_{2k-3}F_{2k-2}-3F_{2k}+2
&>
F_{2k-3}F_{2k-2}-12F_{2k-3}+2 \\
&=
(F_{2k-2}-12)F_{2k-3}+2 .
\end{align*}

Since $k\ge6$, we have $F_{2k-2}\ge F_{10}=55$, and therefore $(F_{2k-2}-12)F_{2k-3}+2>0$. 
This proves the claim.
\end{proof}

\begin{exm}
If $k=5$, then we have
\[
d(\Sigma(F_{11},F_{12},F_{13}))
=
F_{F_{11},F_{12}}(3,4)
=
90
>
88
=
F_{11}-1.
\]
\end{exm}

\begin{rmk}
Assume that $F_{2k+1}$ is odd and $k<5$. 
Then $k=2$ or $3$. 
In particular, $l_{F_{2k+1},F_{2k+2}} \le F_6 = 8$.
Therefore, by Theorem~\ref{thm(1-19)}, we obtain
\[
d(\Sigma(F_{5},F_{6},F_{7}))=F_{5}-1=4
\quad\text{and}\quad
d(\Sigma(F_{7},F_{8},F_{9}))=F_{7}-1=12.
\]
\end{rmk}

\begin{cor}
\label{cor(Fibonacci-cases)}
Suppose that $F_{2k+1}$ is odd and $k>4$. 
Then the Brieskorn homology $3$-sphere $\Sigma(F_{2k+1},F_{2k+2},F_{2k+3})$ is homology cobordant to neither
\[
\Sigma\!\left(F_{2k+1},\frac{3F_{2k+1}-1}{2},3F_{2k+1}+2\right)
\quad\text{nor}\quad
\Sigma(F_{2k+1},2F_{2k+1}-1,2F_{2k+1}+1). 
\]
\end{cor}

\begin{proof}
We compute $s_{F_{2k+1},(3F_{2k+1}-1)/2}=1$ and $s_{F_{2k+1},\,2F_{2k+1}-1}=2$. 
By Theorem~\ref{thm(Example[KS20])}~(3) and~(1), we obtain
\begin{align*}
d(\Sigma(F_{2k+1},(3F_{2k+1}-1)/2,3F_{2k+1}+2))
&= d(\Sigma(F_{2k+1}, 2F_{2k+1}-1,2F_{2k+1}+1))\\
&= F_{2k+1}-1
 = D(F_{2k+1},F_{2k+2},F_{2k+3}).
\end{align*}
On the other hand, Proposition~\ref{prop(Fibonacci-case)} shows that
\[
d(\Sigma(F_{2k+1},F_{2k+2},F_{2k+3}))
>
D(F_{2k+1},F_{2k+2},F_{2k+3}).
\]
Therefore, the $d$-invariant of 
$\Sigma(F_{2k+1},F_{2k+2},F_{2k+3})$
differs from those of
\[
\Sigma(F_{2k+1},(3F_{2k+1}-1)/2,3F_{2k+1}+2)
\quad\text{and}\quad
\Sigma(F_{2k+1}, 2F_{2k+1}-1,2F_{2k+1}+1).
\]
Hence, Remark~\ref{rmk(notcobordant)} implies the claim.
\end{proof}

\begin{cor}
Let $\mathcal{S}
=
\left\{
\Sigma(F_{2n+1},F_{2n+2},F_{2n+3})
\;\middle|\;
n \ge 2 \ \text{and} \ F_{2n+1}\ \text{is odd}
\right\}$. 
Then $\Theta_{\mathbb{Z}}^3$ contains the subgroup 
\[
\bigoplus_{\substack{n \ge 2 \\ F_{2n+1}\ \text{is odd}}}
\mathbb{Z}
\big[
\Sigma(F_{2n+1},F_{2n+2},F_{2n+3})
\big]_{\sim_{\mathbb{Z}}}. 
\]
\end{cor}

\begin{proof}
The argument is identical to that of Corollary~\ref{cor(D=d-partI)}.
\end{proof}

\section{Application to the knot concordance group}
\label{sec(application-knot-conc)}
In this section, we apply our results to the concordance subgroup $\mathcal{C}_{\mathrm{TS}}$. 

Referring to Sections~\ref{sec(infinite-kinds-of-infinite-classes)} and~\ref{sec(d-neq-D)}, Theorem~\ref{thm([End95]reinterpretation)} implies the following.

\begin{thm}
\label{thm(knot-d=D=p-1)}
Let $k$ and $l$ be positive integers. 
Define
\[
(p'_{k,l}, q'_{k,l}, r'_{k,l}) = (4kl+1, 2(2k+1)l+1, (2k+1)(4kl+1)+2k).
\]
For each fixed positive integer $l$, then $\mathcal{C}_{\mathrm{TS}}$ contains the subgroup 
\[
\bigoplus_{k=1}^{\infty}
\mathbb{Z}[P(-p'_{k,l},q'_{k,l},r'_{k,l})],
\]
and for each fixed positive integer $k$, then $\mathcal{C}_{\mathrm{TS}}$ contains the subgroup 
\[
\bigoplus_{l=1}^{\infty}
\mathbb{Z}[P(-p'_{k,l},q'_{k,l},r'_{k,l})].
\]
\end{thm}

\begin{proof}
By definition, $p'_{k,l}$, $q'_{k,l}$, and $r'_{k,l}$ are all odd integers.
A direct computation shows that $1<p'_{k,l}<q'_{k,l}<r'_{k,l}$ and $
p'_{k,l}q'_{k,l}
+ p'_{k,l}r'_{k,l}
- q'_{k,l}r'_{k,l}
= 1$. 
Moreover, for any positive integer $k$, then 
$
p'_{k,l}q'_{k,l}r'_{k,l}
<
p'_{k+1,l}q'_{k+1,l}r'_{k+1,l}$. 
Similarly, for any positive integer $l$, then 
$
p'_{k,l}q'_{k,l}r'_{k,l}
<
p'_{k,l+1}q'_{k,l+1}r'_{k,l+1}
$. 
Therefore, the conclusion follows from
Theorem~\ref{thm([End95]reinterpretation)}.
\end{proof}

\begin{rmk}
If $(k_1,l_1)\neq (k_2,l_2)$, then $(p'_{k_1,l_1},q'_{k_1,l_1}) \neq (p'_{k_2,l_2},q'_{k_2,l_2})$.
Since there are infinitely many choices of either $k$ or $l$,
Theorem~\ref{thm(knot-d=D=p-1)} produces infinitely many explicit
$\mathbb{Z}^{\infty}$-subgroups of $\mathcal{C}_{\mathrm{TS}}$
satisfying
\[
d(\Sigma(p,q,r)) = D(p,q,r)
\quad \text{and} \quad
s_{p,q}\in(0,1).
\]
\end{rmk}

\begin{thm}
\label{thm(knot-d=D>p-1)}
Let $u$ and $v$ be positive integers. 
Define
\[
(p'_{u,v}, q'_{u,v}, r'_{u,v}) = (4u(u+1)v-2u-1, 4u(2u+1)v-4u-1, 4(u+1)(2u+1)v-4u-3). 
\]
For each fixed positive integer $v$, then $\mathcal{C}_{\mathrm{TS}}$ contains the subgroup 
\[
\bigoplus_{u=1}^{\infty}
\mathbb{Z}[P(-p'_{u,v},q'_{u,v},r'_{u,v})], 
\]
and for each fixed positive integer $u$, then $\mathcal{C}_{\mathrm{TS}}$ contains the subgroup 
\[
\bigoplus_{v=1}^{\infty}
\mathbb{Z}[P(-p'_{u,v},q'_{u,v},r'_{u,v})]. 
\]
\end{thm}

\begin{proof}
The argument is identical to that Theorem \ref{thm(knot-d=D=p-1)}. 
\end{proof}

\begin{rmk}
If $(u_1,v_1)\neq (u_2,v_2)$, then $(p'_{u_1,v_1},q'_{u_1,v_1}) \neq (p'_{u_2,v_2},q'_{u_2,v_2})$.
Since there are infinitely many choices of either $u$ or $v$,
Theorem~\ref{thm(knot-d=D>p-1)} produces infinitely many explicit
$\mathbb{Z}^{\infty}$-subgroups of $\mathcal{C}_{\mathrm{TS}}$
satisfying
\[
d(\Sigma(p,q,r)) = D(p,q,r)
\quad \text{and} \quad
s_{p,q}\in(1,2).
\]
\end{rmk}

\begin{thm}
\label{thm(knot-d>D)}
Let $t$ be a positive integer and $k$ a nonnegative integer. \\
Define $(p'_{t,k,1}, q'_{t,k,1}, r'_{t,k,1}) = (4t(2t+1)(2k+1)+4t+1, (2t+1)(6t+1)(2k+1)+6t+2, 4t(6t+1)(2k+1)+12t+1)$ and $(p'_{t,k,2}, q'_{t,k,2}, r'_{t,k,2}) = (2t(2t-1)k+4t-1, 2t(3t-1)k+6t-1, 2(2t-1)(3t-1)k+12t-5)$. 
For each fixed positive integer $t$, then $\mathcal{C}_{\mathrm{TS}}$ contains the subgroup 
\[
\bigoplus_{k=0}^{\infty}
\mathbb{Z}[P(-p'_{t,k,i},q'_{t,k,i},r'_{t,k,i})], 
\]
and for each fixed nonnegative integer $k$, then $\mathcal{C}_{\mathrm{TS}}$ contains the subgroup 
\[
\bigoplus_{t=1}^{\infty}
\mathbb{Z}[P(-p'_{t,k,i},q'_{t,k,i},r'_{t,k,i})], 
\]
for each $i\in\{1,2\}$.
\end{thm}

\begin{proof}
The argument is identical to that of Theorem~\ref{thm(knot-d=D=p-1)}.
\end{proof}

\begin{rmk}
If $(t_1,k_1,i_1)\neq(t_2,k_2,i_2)$, then $(p'_{t_1,k_1,i_1}, q'_{t_1,k_1,i_1})\neq(p'_{t_2,k_2,i_2}, q'_{t_2,k_2,i_2})$.
Since there are infinitely many choices of either $t$ or $k$,
Theorem~\ref{thm(knot-d>D)} produces infinitely many explicit $\mathbb{Z}^{\infty}$-subgroups of $\mathcal{C}_{\mathrm{TS}}$ satisfying
\[
d(\Sigma(p,q,r)) > D(p,q,r)
\quad \text{and} \quad
s_{p,q}\in(1,2). 
\]
\end{rmk}

The results obtained in this section naturally lead to the following questions.

\begin{que}
Does there exist an infinite sequence
$((p_{n,k},q_{n,k},r_{n,k}))_{n,k=1}^{\infty}$
and a submodule
$A$ of $\mathcal{C}_{\mathrm{TS}}$
such that
\[
\mathcal{C}_{\mathrm{TS}}
=
A \oplus
\bigoplus_{n=1}^{\infty}
\bigoplus_{k=1}^{\infty}
\mathbb{Z}[P(-p_{n,k},q_{n,k},r_{n,k})], \quad \text{and}
\]
\[
d(\Sigma(p_{n,k_1},q_{n,k_1},r_{n,k_1}))
=
d(\Sigma(p_{n,k_2},q_{n,k_2},r_{n,k_2}))
\]
for all positive integers $k_1$ and $k_2$?
\end{que}

\begin{que}
Does there exist an infinite sequence $((p_n,q_n,r_n))_{n=1}^{\infty}$ such that \\
$(P(-p_n,q_n,r_n))_{n=1}^{\infty}$ generates a $\mathbb{Z}^{\infty}$-summand of
$\mathcal{C}_{\mathrm{TS}}$ and
\[
d(\Sigma(p_n,q_n,r_n))
=
D(p_n,q_n,r_n)
>
p_n-1 \quad \text{for all $n$?}
\]
\end{que}

\begin{que}
Does there exist an infinite sequence $((p_n,q_n,r_n))_{n=1}^{\infty}$ such that \\
$(P(-p_n,q_n,r_n))_{n=1}^{\infty}$ generates a $\mathbb{Z}^{\infty}$-summand of
$\mathcal{C}_{\mathrm{TS}}$ and 
\[
d(\Sigma(p_n,q_n,r_n))
>
D(p_n,q_n,r_n) \quad \text{for all $n$?}
\]
\end{que}

\bibliographystyle{amsalpha}
\bibliography{references.bib}

\end{document}